\documentclass[12pt,a4paper, oneside, reqno]{amsart}
\usepackage[foot]{amsaddr} 
\usepackage{mathtools}

\usepackage{amsmath, nicefrac, amsthm, verbatim, amsfonts, amssymb, xcolor, bbm}
\usepackage{graphics, xspace, enumerate}
\usepackage{enumitem}
\usepackage{stix}
\usepackage[top=30mm,bottom=30mm,left=25mm,right=25mm]{geometry}

\usepackage[colorinlistoftodos,  textsize=tiny]{todonotes} 					 
\presetkeys{todonotes}{fancyline, color=blue!30}{}

\usepackage[bb=boondox]{mathalfa}
\usepackage{graphicx}
\usepackage[colorlinks=true,citecolor=red,urlcolor=blue,linkcolor=red,bookmarksopen=true,unicode=true,pdffitwindow=false]{hyperref}
\usepackage[english]{babel}
\usepackage[languagenames,fixlanguage]{babelbib}
\hypersetup{pdfauthor={}}
\hypersetup{pdftitle={FLT for the capacity of the range of random walks}}

\theoremstyle{plain}
\newtheorem{theorem}{Theorem}[section]
\newtheorem{corollary}[theorem]{Corollary}
\newtheorem{lemma}[theorem]{Lemma}
\newtheorem{proposition}[theorem]{Proposition}
\theoremstyle{definition}
\newtheorem{remark}[theorem]{Remark}

\newcommand {\Prob} {\ensuremath{\mathbb{P}}}
\newcommand {\R} {\ensuremath{\mathbb{R}}}

\newcommand {\N} {\ensuremath{\mathbb{N}}}

\newcommand {\distJ}{\mathrm{dist}_{\mathcal{J}_1}}

\newcommand{\cl}{\mathrm{cl}}

\newcommand{\Var}{\operatorname{Var}}

\newcommand{\bbN}{\mathbb{N}}

\newcommand{\bbR}{\mathbb{R}}

\newcommand{\bbE}{\mathbb{E}}

\newcommand{\bY}{\mathbf{Y}}
\newcommand{\bX}{\mathbf{X}}
\newcommand{\bS}{\mathbf{S}}
\newcommand{\SRange}{\{\bS(0),\ldots ,\bS(n)\}}
\newcommand{\conv}{\mathrm{conv}}

\numberwithin{equation}{section}

\title[Convex hulls of stable random walks]{Convex hulls of stable random walks}

\author[W.\ Cygan]{Wojciech Cygan}
\address[Wojciech Cygan]{Technische Universit\"{a}t Dresden\\ Faculty of Mathematics\\ Institute of Mathematical Stochastics\\Zellescher Weg 25, 01069 Dresden, Germany
\& 
University of Wroc\l{}aw\\
Faculty of Mathematics and Computer Science\\
Institute of Mathematics\\pl.\ Grunwaldzki 2/4\\ 50--384 Wroc\l{}aw\\ Poland}
\email{wojciech.cygan@uwr.edu.pl}

\author[N.\ Sandri\'{c}]{Nikola Sandri\'{c}}
\address[Nikola\ Sandri\'{c}]{Department of Mathematics\\University of Zagreb\\ Zagreb\\Croatia}
\email{nsandric@math.hr}

\author[S.\ \v{S}ebek]{Stjepan\ \v{S}ebek}
\address[Stjepan\ \v{S}ebek]{Department of Applied Mathematics\\
	Faculty of Electrical Engineering and Computing\\
	University of Zagreb\\ 
 Zagreb\\ 
	Croatia}
\email{stjepan.sebek@fer.hr}


\begin{document}
\allowdisplaybreaks[4]

\begin{abstract}
We consider convex hulls of random walks whose steps belong to the domain of attraction of a stable law in $\bbR^d$. We prove convergence of the convex hull in the space of all convex and compact subsets of $\bbR^d$, equipped with the Hausdorff distance, towards the convex hull spanned by a path of the limit stable L\'evy process. As an application, we establish convergence of (expected) intrinsic volumes under some mild moment/structure assumptions posed on the random walk. 
\end{abstract}

\maketitle

\section{Introduction}
Geometric aspects of random walks have become an attractive topic in modern probability theory, especially questions concerning the shape of convex polygons spanned by points lying on a path of a planar random walk inspired a vast number of authors. The seminal article in this direction is  \cite{Spitzer-Widom} where Spitzer and Widom found a combinatorial formula for the expected value of the perimeter of the convex hull of an arbitrary random walk in the plane. To prove this result they applied techniques from integral geometry while Baxter  \cite{Baxter} provided a purely combinatorial argument.
Later Snyder and Steele \cite{Snyder-Steele} found bounds for the variance of the perimeter, see also \cite{Nielsen-Baxter_Area} for the expected area of the convex hull.  More recently, plenty of interesting results have been obtained for planar random walks with finite variance \cite{McRedmond-Wade_EJP}, \cite{Wade-Xu_PAMS}, \cite{Wade-Xu_SPA} and for multidimensional random walks  \cite{Kabluchko-Vys-Zap_Adv}, \cite{Kabluchko-Vys-Zap_GAFA}, \cite{McRedmond_survey}, \cite{Vysotsky}. We refer to \cite{Majumdar_EvS} for a historical perspective and a handy presentation of available results for convex hulls of random walks.

In this article, we consider a class of stable random walks in the multidimensional case. 
Let $\{\bY_i\}_{i\in\N}$ be a sequence of independent and identically distributed $\R^d$-valued random vectors defined on a given probability space $(\Omega, \mathcal{A},\mathbb{P})$
 and let $\bS(n)=\sum_{i=1}^n \bY_i$ be the corresponding  random walk such that $\bS(0)=0$. We assume that the distribution of the step $\bY_1$ belongs to the domain of attraction  of an $\alpha$-stable law in $\R^d$ with index of stability $\alpha\in(0,2]$. This means that 
there are sequences $\{b_n\}_{n\ge0}\subset(0,\infty)$ and $\{\mathbf{a}_n\}_{n\ge0}\subset\R^d$, such that 
\begin{equation}\label{DOM-ATT}
\frac{\bS(n)-\mathbf{a}_n}{b_n}\xrightarrow[n\nearrow\infty]{\mathcal{D}} \bX(1),
\end{equation}
where $\{\bX(t)\}_{t\ge0}$ is an $\alpha$-stable L\'evy process in $\R^d$. The notation\, $\xrightarrow{\mathcal{D}}$ is used for convergence in distribution of random vectors in $\R^d$.

The main object of our study is the convex hull generated by the set $\{\bS(0),\bS(1),\ldots ,\bS(n)\}$ which we denote by $\mathrm{conv} \{\bS(0),\bS(1),\ldots ,\bS(n)\}$. 
So far not much is known about its limit behaviour if the step distribution of the walk is heavy-tailed. Asymptotic expansions of the mean perimeter and area of $\conv  \SRange$ were found in \cite{Grebenkov} for planar random walks whose steps have stable-like symmetric continuous densities. There are a few articles on convex hulls of stable (and L\'{e}vy) processes in $\R^d$ which are strongly related to our work, see  \cite{Kampf-Molchanov} , \cite{Molchanov_JMA} and \cite{Molchanov-Wespi}.  We also refer to \cite{Eldan} for the case of the convex hull of Brownian motion, to \cite{Alsmeyer} for the treatment on convex minorants of one-dimensional stable random walks, and 
to \cite{Hull-Reg-Var_proc} for convergence of convex hulls of point processes under low-moment assumptions. 

Our first goal is to enlighten the connection between convex hulls of discrete and continuous-time processes with infinite variance and to understand under which scaling procedure we can approximate the convex hull of the limit process with those of random walks.
We thus investigate convergence of convex hulls of stable random walks at the level of sets in the space of all convex and compacts subsets of $\bbR^d$ (so-called convex bodies) equipped with the Hausdorff distance. Such convergence holds for all stable walks that have infinite first moment (case $\alpha <1$). For random walks attracted by a Cauchy law (case $\alpha =1$) our methods apply only if the walk is symmetric. For stable random walks with finite expectation (case $1< \alpha \le 2$) we distinguish between two cases, that is when the drift is zero or not. In the case with zero drift (and for the cases mentioned above) the convex hull of the walk, rescaled with the sequence $b_n$ from \eqref{DOM-ATT}, converges to the convex hull generated by a path of the limit process $\{\bX(t)\}_{t\ge0}$ run up to time one. If there is a non-zero drift such scaling turns out to be inappropriate.
In this context, we learned much about the scaling limits of convex hulls and how to handle the non-zero drift case from articles \cite{McRedmond_survey} and \cite{Wade-Xu_SPA}.
We apply techniques developed in these both articles to adjust the scaling along the drift and we obtain corresponding convergence of rescaled convex hulls. 

There are some basic geometric functionals which enable us to describe and study the shape of a convex body in the Euclidean space. In the plane these are perimeter and area (also diameter) whereas in higher dimensions surface area and volume play the key role. 
The other important quantities are involved in the celebrated Steiner formula which shows how fast grows the volume of a convex body which is expanded by a rescaled unit ball. This volume is represented as a polynomial of degree $d$ and its coefficients are (up to a constant) the so-called intrinsic volumes of the convex body. They can be regarded as a generalization of surface area and volume as the pre-last and last intrinsic volumes actually coincide with those two functionals and the first intrinsic volume can be seen as a counterpart of perimeter. These geometric quantities have been recently studied also for convex hulls of random walks and continuous time stochastic processes. 
We were inspired by article \cite{Vysotsky} where the authors established precise formulas for the expected value of intrinsic volumes of $\conv \SRange$ under a general position assumption on the walk which means that it does not stay almost surely in any affine hyperplane of $\R^d$, 
see also \cite{Nielsen-Baxter_Area}, \cite{Kabluchko-Vys-Zap_Adv} and \cite{Kabluchko-Vys-Zap_GAFA}. 
Another motivation for our work came from article \cite{Molchanov-Wespi} where the authors computed expected intrinsic volumes of the convex hull for L\'{e}vy processes, see also \cite{Kampf-Molchanov} and \cite{Molchanov_JMA}  for the case of symmetric stable processes. 

In this article, we establish convergence of intrinsic volumes of the convex hull of stable random walks. If the walk has no drift (or $\alpha<1$, or $\alpha=1$ and $\{\bS(n)\}_{n\ge0}$ is symmetric) then by a continuity argument we infer that the $m$-th intrinsic volume of $\conv \SRange$, scaled by $b_n^m$, converges towards the corresponding intrinsic volume of the hull spanned by a path of the limit process $\{\bX(t)\}_{t\ge0}$. Similarly as before, if the drift does not vanish an another scaling has to be  employed which is found by an approach based upon \cite{Wade-Xu_SPA}.

The main part of the article is devoted to convergence of the expected intrinsic volumes of the convex hull.
 In the zero-drift case we show that under the general position assumption 
the expected $m$-th intrinsic volume, scaled by $b_n^m$, converges to a limit 
given in terms of the Gram determinant spanned by independent copies of the process $\{\bX(t)\}_{t\ge0}$. If the process is symmetric this determinant can be computed through techniques developed in \cite{Molchanov_JMA} and \cite{Molchanov-Wespi}. Moreover, via our methods we obtain the result that (even for a non-symmetric stable process) the first expected intrinsic volume of its convex hull is determined by its first absolute moment. 
For non-zero drift random walks we need finiteness of absolute moments of $\bY_1$ of higher order which implies that $\alpha =2$ and $\{\bX(t)\}_{t\ge0}$  is a Brownian motion. Then, appropriately rescaled, $m$-th expected intrinsic volume of the convex hull of the walk converges to the convex hull spanned by a time-space Brownian motion which is constructed from the original process $\{\bX(t)\}_{t\ge0}$. As an additional result, we establish a closed formula for the expected volume of the convex hull of the time-space Brownian motion. 

Finally, we study asymptotics of the variance of intrinsic volumes of the convex hull of random walks with finite moments of order higher than two and under the general position assumption. For zero-drift random walks we find an appropriate scaling for all $m\in\{1,\ldots ,d\}$, while if there is a non-zero drift we only cover the case $m\ge2$. For the first intrinsic volume we establish an upper bound with the term of linear order.

The article is organized as follows. We start by a brief discussion on necessary definitions and results from geometry and probability. In Section \ref{sec:Hulls} we present convergence of convex hulls in the space of convex bodies as well as convergence of their  intrinsic volumes. Section \ref{sec:Means} is devoted to convergence of mean intrinsic volumes while in Section \ref{sec:Var} we focus on variance asymptotics.

\subsection*{Stable random walks}
We use notation
$\bY_1=(Y^{(1)}_1,\dots,Y^{(d)}_1)$ 
 and $\bS(n)=(S^{(1)}(n),\dots,S^{(d)}(n))$.
We remark that \eqref{DOM-ATT} and continuous mapping theorem  imply that for any $k=1,\dots, d$, 
\begin{equation}\label{COOR-DOM-ATT}
\frac{S^{(k)}(n)-a^{(k)}_n}{b_n}\xrightarrow[n\nearrow\infty]{\mathcal{D}} X^{(k)}(1),
\end{equation} 
where $\mathbf{a}_n=(a^{(1)}_n,\dots,a_n^{(d)})$ and $\bX(t)=(X^{(1)}(t),\dots,X^{(d)}(t)).$ 
It follows that each $\{X^{(k)}(t)\}_{t\ge0}$ is a one-dimensional $\alpha$-stable L\'evy process and whence the coordinates of $\mathbf{a}_n$ are given by
\begin{equation}\label{a_n}
a_n^{(k)}=\begin{cases}
0, & \alpha<1,\\
nb_n\mathbb{E}\bigl[\sin(Y^{(k)}_1/b_n)\bigr], & \alpha=1,\\
n\,\mathbb{E}[Y^{(k)}_1],& \alpha>1.
\end{cases}
\end{equation}
Moreover, $b_n=n^{1/\alpha}\ell(n)$ for a function $\ell(u)$ which is slowly varying at infinity.	
We observe that \cite[Theorem 2.7]{Skorohod-1957} and \eqref{COOR-DOM-ATT} imply the following convergence
$$
\left\{\frac{S^{(k)}(\lfloor nt\rfloor)-a^{(k)}_{\lfloor nt\rfloor}}{b_n}\right\}_{t\ge0}\xrightarrow[n\nearrow\infty]{\mathcal{J}^{1}_1}\{X^{(k)}(t)\}_{t\ge0}.
$$ 
This, together with \cite[Problem 5.9]{Billingsley-book}, gives that $\{(\bS(\lfloor nt\rfloor)-\mathbf{a}_{\lfloor nt\rfloor})/b_n\}_{t\ge0}$ is tight. By
$\mathcal{J}_1^d$ we denote the Skorokhod topology on the space of $\R^d$-valued c\`{a}dl\`{a}g functions.
It is straightforward to check that finite-dimensional distributions of $\{(\bS(\lfloor nt\rfloor)-\mathbf{a}_{\lfloor nt\rfloor})/b_n\}_{t\ge0}$ converge to those of $\{\bX(t)\}_{t\ge0}$ (see e.g. \cite[page 88]{Billingsley-book}). It follows
\begin{equation}\label{FCLT}
\left\{\frac{\bS(\lfloor nt\rfloor)-\mathbf{a}_{\lfloor nt\rfloor}}{b_n}\right\}_{t\ge0}\xrightarrow[n\nearrow\infty]{\mathcal{J}^{d}_1}\{\bX(t)\}_{t\ge0}.
\end{equation} 
We point out  that  \eqref{DOM-ATT} and \eqref{FCLT} are actually equivalent, see \cite[Proposition VI.3.14]{Jacod}.
For a detailed discussion on multidimensional stable laws and their domains of attraction we refer to \cite{Samorodnitsky}, see also \cite{Resnick_JMultivAnal} and \cite{Rvaceva}.

\subsection*{Geometric issues and c\`{a}dl\`{a}g paths}
We write $\Vert x \Vert$ for the Euclidean norm of $x\in \R^d$ and 
let $\mathbb{S}^{d-1}=\{x\in \R^d: \Vert x\Vert =1\}$ be the unit sphere in $\R^d$. For $x\in\R^d$ and $A\subset\R^d$ the distance from $x$ to $A$ is defined as $\mathrm{dist}(x,A)=\inf_{y\in A}\Vert x-y\Vert$. For a bounded set $A\subset\R^d$ we denote by $\mathrm{cl}\, A$ the closure of $A$.
For two bounded sets $A,B\subset\R^d$, let $\mathrm{dist}_\mathcal{H}(A,B)$ be the Hausdorff  distance between $A$ and $B$ defined as
\begin{align*}
\mathrm{dist}_\mathcal{H}(A,B)=\max\big\{\sup_{x\in B}\mathrm{dist}(x,A),\sup_{x\in A}\mathrm{dist}(x,B)\big\}.
\end{align*}
Note that on the family of bounded sets this defines merely a pseudo-metric, while on the family of compact sets it becomes a metric.

Let $\mathcal{K}^d$ denote the family of all convex bodies in $\R^d$  and let $\mathcal{K}_0^d=\{A\in \mathcal{K}^d:0\in A\}$.
By $\mathrm{conv}\, A$ we denote the convex hull of the set $A$. 
The space $\mathcal{K}^d$ equipped with the topology generated by the Hausdorff metric becomes a complete metric space. Our main reference for convex geometry is \cite{Schneider-book}.

Let $\mathcal{D}^d_0[0,1]$ be the space of c\`{a}dl\`{a}g functions $f:[0,1]\to\R^d$ with $f(0)=0.$ We equip it with the standard $\mathcal{J}_1^d$ Skorokhod topology defined  through the metric
\begin{align*}
\distJ(f,g)= \inf_{\lambda \in \Lambda}\{ \Vert f-g\circ \lambda\Vert_\infty + \Vert \lambda - \mathrm{Id}\Vert_\infty\},
\end{align*}
where $\Lambda$ is the set of all increasing bijections $\lambda$ of $[0,1]$ such that $\lambda$ and its inverse $\lambda^{-1}$ are both continuous, and $\mathrm{Id}$ is the identity map on $[0,1]$.
To any $f\in \mathcal{D}^d_0[0,1]$ we associate the   convex hull of the closure of its path, that is 
\begin{align*}
f\mapsto H(f)=\mathrm{conv}f[0,1],
\end{align*} 
where $f[0,1] = \mathrm{cl }\{f(t):t\in[0,1]\}$. We remark that 
$$\mathrm{conv}f[0,1] = \mathrm{cl}\, \mathrm{conv}\{f(t)\colon t\in [0,1]\},
$$
see \cite[Proposition 3.2]{Gruber_book}.
The following continuity result is a crucial observation for our purposes. It has been already found in \cite{Molchanov-Wespi} and applied in the context of convex hulls of stable L\'evy processes.

\begin{lemma}\label{H-ctns}
The mapping $H\colon (\mathcal{D}^d_0[0,1], \mathrm{dist}_{\mathcal{J}_1}) \to (\mathcal{K}_0^d, \mathrm{dist}_\mathcal{H})$ is Lipschitz continuous.
	\end{lemma}
\begin{proof}

	For any two bounded sets $A,B\subset\R^d$ it holds that 
	$\mathrm{dist}_\mathcal{H}(\mathrm{cl}\, A,\mathrm{cl}\, B)= \mathrm{dist}_\mathcal{H}(A,B)$ and 
$$	
	\mathrm{dist}_\mathcal{H}(\mathrm{conv} A,\mathrm{conv} B)\le \mathrm{dist}_\mathcal{H}(A,B),$$
	see e.g.\ \cite[Lemmas 3.2.2 and 3.3.1]{McRedmond_PhD_thesis}.
	 It follows 
	 $$
	\mathrm{dist}_\mathcal{H}\bigl(H(f),H(g)\bigr)\le\mathrm{dist}_\mathcal{H}\bigl(f[0,1],g[0,1]\bigr).
	$$ 
	Finally, we remark that
	$$
	\mathrm{dist}_\mathcal{H}\bigl(H(f),H(g)\bigr)\le\mathrm{dist}_{\mathcal{J}_1}(f,g),$$
	see e.g.\ \cite[Lemma 3.2.1]{McRedmond_PhD_thesis}.
	\end{proof}

We recall definitions of a few basic geometric functionals which can be defined for any convex body. Later we  investigate them for convex hulls of stable random walks. For $A\in \mathcal{K}^d$ its support function is defined by $$s_A(x)=\sup_{y\in A}\langle x,y\rangle,$$ where $\langle\cdot,\cdot \rangle$ stands for the standard scalar product in $\R^d$. The mean width and Steiner point of $A\in \mathcal{K}^d$ are respectively defined as
\begin{equation}\label{MW}
\mathsf{w}(A)=\frac{2}{\varpi_d}\int_{\mathbb{S}^{d-1}}s_A(\theta)\sigma (\mathrm{d}\theta)\qquad\text{and}\qquad 
\mathsf{p}(A)=\frac{1}{\kappa_d}\int_{\mathbb{S}^{d-1}}s_A(\theta)\theta \sigma (\mathrm{d}\theta),
\end{equation} 
where $\sigma(\mathrm{d}\theta)$ is the surface measure over $\mathbb{S}^{d-1}$ and
$\varpi_d= \sigma (\mathbb{S}^{d-1})$, while
$\kappa_{d} = \mathrm{Vol}_{d} (B^{d})$ is the $d$-dimensional Lebesgue measure of the unit ball $B^d=\{x\in\R^d:\|x\|\le1\}$. The both numbers are given by
$
\kappa_{d}  = \pi^{d/2}/\Gamma \big(1+\frac{d}{2}\big)
$
and
$
\varpi_d =d\kappa_d.
$
 For any $\rho\ge0$ the outer parallel body of $A$ at distance $\rho$ is defined as $A+\rho B^d$. The classical Steiner formula provides an expansion for its $d$-dimensional Lebesgue measure 
in terms of a polynomial of degree at most $d$ whose coefficients are important geometric quantities. The polynomial is of the form
\begin{align}\label{Steiner}
\mathrm{Vol}_d\bigl(A+\rho B^d\bigr)=\sum_{m=0}^d \rho^{d-m} \kappa_{d-m}
V_m(A),
\end{align}
where 
$V_0(A),\ldots ,V_{d}(A)$ are so-called intrinsic volumes of the set $A$.
 It is known that $V_0(A)=1$ and  $V_1(A)$ is proportional to the mean width of $A$, that is 
\begin{align}\label{V_1=mean_width}
V_1(A)=\frac{d\kappa_d}{2\kappa_{d-1}}\mathsf{w}(A).
\end{align}
We remark that if $d=2$ then $V_1(A)$ is equal to one half of the perimeter of $A$.
Furthermore, $V_{d-1}(A)$ is equal to one half of the surface area of $A$ and $V_{d}(A)=\mathrm{Vol}_d(A)$. 
 We remark that $A\mapsto V_m(A) $, $m\in \{0,1,\ldots ,d\}$, and $A\mapsto \mathsf{p}(A) $  are  continuous  mappings from $(\mathcal{K}^d,\mathrm{dist}_{\mathcal{H}})$ to $([0,\infty),|\cdot|)$ and from $(\mathcal{K}^d,\mathrm{dist}_{\mathcal{H}})$ to $(\R^d, \Vert\! \cdot\! \Vert)$, respectively, see \cite[Theorem III.1.1]{Schneider} and \cite[Lemma 1.8.14]{Schneider-book}.


\section{Convergence of convex hulls}\label{sec:Hulls}
We start by showing a general weak-convergence result for convex hulls of stable random walks in the space of convex bodies equipped with the Hausdorff distance. 
The scaling of order $b_n$ coming from assumption \eqref{DOM-ATT} is appropriate if there is no centering. We thus investigate separately the case when the walk has finite non-zero expectation. 
The following result should be compared with \cite[Theorem 2.5 and Corollary 2.6]{Wade-Xu_SPA}.

\begin{proposition}\label{FCLT-HULL}
	Under assumption \eqref{DOM-ATT} it holds 
	$$
	\frac{\mathrm{conv}\{\bS(0)-\mathbf{a}_0,\ldots,\bS(n)-\mathbf{a}_n\}}{b_n}\xRightarrow[n\nearrow\infty]{} \mathrm{conv}\,\bX[0,1]
	$$
	 in the sense of weak convergence in  $(\mathcal{K}_0^d,\mathrm{dist}_{\mathcal{H}})$. 
We also have for $m\in\{1,\ldots ,d\}$,
$$
\frac{V_m\bigl(\mathrm{conv}\{\bS(0)-\mathbf{a}_0,\ldots,\bS(n)-\mathbf{a}_n\}\bigr)}{b^m_n}\xrightarrow[n\nearrow\infty]{\mathcal{D}} 
V_m \bigl(\mathrm{conv}\, \bX[0,1] \bigr)
$$
and 
$$\frac{\mathsf{p}\bigl(\mathrm{conv}\{\bS(0)-\mathbf{a}_0,\ldots,\bS(n)-\mathbf{a}_n\}\bigr)}{b_n}\xrightarrow[n\nearrow\infty]{\mathcal{D}} 
\mathsf{p} \bigl(\mathrm{conv}\, \bX[0,1] \bigr).
$$
\end{proposition}
	
\begin{proof} 
	Observe first that
 for any bounded set $A\subset\R^d$, any  $\mathbf{a}\in\R^d$ and any $M\in\R^{d\times d}$, we have 
	$\mathrm{conv}(A+\mathbf{a})=\mathrm{conv}(A)+\mathbf{a},\ \mathrm{conv}(M\cdot A)=M\cdot\mathrm{conv}(A),\ \cl( A+\mathbf{a})=\cl(A)+\mathbf{a}$ and $\cl(M\cdot A)=M\cdot\cl (A).$
	It follows
	$$
	H\left(\bigl\{(\bS(\lfloor nt\rfloor)-\mathbf{a}_{\lfloor nt\rfloor})/b_n\bigr\}_{t\in[0,1]}\right)=\frac{\mathrm{conv}\{\bS(0)-\mathbf{a}_0,\ldots,\bS(n)-\mathbf{a}_n\}}{b_n}.
	$$ 
	Finally, \eqref{FCLT} together with continuous mapping theorem implies
	$$
	\frac{\mathrm{conv}\{\bS(0)-\mathbf{a}_0,\ldots,\bS(n)-\mathbf{a}_n\}}{b_n}\xRightarrow[n\nearrow\infty]{}H\bigl(\{\bX(t):t\in[0,1]\}\bigr),
	$$
	as desired. 
	By employing again continuous mapping theorem and the facts that for $A\in\mathcal{K}^d$ and $a\in\R_+$ (see the discussion after  \cite[Theorem III.1.1]{Schneider}) it holds $$ V_m(a\cdot A)=a^mV_m(A)\qquad \text{and}\qquad \mathsf{p}(a\cdot A)=a\mathsf{p}(A),$$  we infer the last two formulas.
\end{proof}

\begin{remark}\label{REM1}
(i) Proposition \ref{FCLT-HULL} provides information on the limit behaviour for the convex hull $\mathrm{conv}\SRange$ only in the case when $\mathbf{a}_n\equiv 0$. This is true if $\alpha <1$ and if the walk has zero drift for $\alpha>1$. For $\alpha =1$ it covers the case when the walk is symmetric, cf.\ \eqref{a_n}. The non-symmetric case for $\alpha =1$ is not handled in the present article. \\
(ii) If we assume that $\bbE[\lVert \bY_1\rVert^2]<\infty$, $\mu=0$ and define $\Sigma=\bbE[\bY_1 \bY_1^t]$, then $b_n=\sqrt{n}$ (see \cite[Theorem 2.6.6]{Ibragimov-Linnik}) and  
$$
\{\bX(t)\}_{t\ge0}\stackrel{\mathcal{D}}{=}\{\Sigma^{1/2}\mathbf{B}(t)\}_{t\ge0},
$$ 
where $\{\mathbf{B}(t)\}_{t\ge0}$ is the standard $d$-dimensional Brownian motion. In view of Proposition \ref{FCLT-HULL} we obtain
 $$
 \frac{V_m\bigl(\mathrm{conv}\{\bS(0),\ldots,\bS(n)\}\bigr)}{n^{m/2}}\xrightarrow[n\nearrow\infty]{\mathcal{D}} 
 V_m \bigl(\Sigma^{1/2}\mathrm{conv}\, \mathbf{B}[0,1] \bigr).
 $$ In particular,
 $$\frac{V_d\bigl(\mathrm{conv}\{\bS(0),\ldots,\bS(n)\}\bigr)}{n^{d/2}}\xrightarrow[n\nearrow\infty]{\mathcal{D}} 
 \sqrt{\det(\Sigma)}V_d \bigl(\mathrm{conv}\, \mathbf{B}[0,1] \bigr).$$
\end{remark}

\subsection{Non-zero drift case}
In this paragraph, we assume that
 $\alpha>1$ which implies that the first moment of $\bY_1$ is finite.
We consider the case when $\mu=\mathbb{E}[\bY_1]\neq 0$. 
We use notation
$$
V_m(n)=V_m\bigl(\mathrm{conv}\{\bS(0),\ldots,\bS(n)\}\bigr)
\qquad\text{and}\qquad 
P(n)=\mathsf{p}\bigl(\mathrm{conv}\{\bS(0),\ldots,\bS(n)\}\bigr).$$ 
We first establish an almost sure convergence for $V_1(n)$ and $P(n)$. We will use it later to obtain convergence of means in Section \ref{sec:Means}. We remark that almost sure convergence of $V_1(n)$ was established in \cite[Theorem 6.11]{McRedmond_survey}  and \cite[Theorem 1.1]{McRedmond-Wade_EJP} for planar random walks with finite first moment, see also \cite{Snyder-Steele} for random walks with finite variance.

\begin{theorem}\label{V_1_p_drift}
Assume \eqref{DOM-ATT}. Let $\alpha >1$ and $\mu\neq 0$.
It holds
\begin{align*}
\frac{V_1(n)}{n}
\xrightarrow[n\nearrow\infty]{\mathbb{P}\text{-a.s.}} \Vert \mu\Vert
\qquad
\mathrm{and}\qquad
\frac{P(n)}{n}
\xrightarrow[n\nearrow\infty]{\mathbb{P}\text{-a.s.}} \frac{\mu}{2}.
\end{align*} 
\end{theorem}
\begin{proof}
According to \cite[Theorem 3.4]{McRedmond_survey}, the following convergence holds in $ (\mathcal{D}^d_0[0,1],\mathrm{dist}_{\mathcal{J}_1})$, 
$$
\left\{\frac{\bS(\lfloor nt\rfloor)}{n}\right\}_{t\in[0,1]}\xrightarrow[n\nearrow\infty]{\mathbb{P}\text{-a.s.}}\{ t \mu \}_{t\in[0,1]}.
$$
From Lemma \ref{H-ctns} it then follows that in the space $(\mathcal{K}_0^d,\mathrm{dist}_{\mathcal{H}})$,  
$$
\frac{\mathrm{conv}\{\bS(0),\ldots,\bS(n)\}}{n}\xrightarrow[n\nearrow\infty]{\mathbb{P}\text{-a.s.}} [0,\mu],
$$  
where $ [0,\mu]=\{s \mu : s\in[0,1]\}\subset \mathbb{R}^d$. 
Thus, continuous mapping theorem yields
\begin{equation}\label{CONV-DRIFT-1}
\frac{V_1(n)}{n}\xrightarrow[n\nearrow\infty]{\mathbb{P}\text{-a.s.}}V_1( [0,\mu])\qquad\text{and}\qquad \frac{P(n)}{n}\xrightarrow[n\nearrow\infty]{\mathbb{P}\text{-a.s.}} \mathsf{p}( [0,\mu]).
\end{equation}
We claim that $V_1([0,\mu])=\Vert \mu\Vert$. Indeed, we obtain by Steiner formula \eqref{Steiner} that for any $\rho \geq 0$, 
\begin{align*}
\mathrm{Vol}_d\bigl([0,\mu]+\rho B^d\bigr)&=\kappa_d\rho^dV_0\bigl([0,\mu]\bigr)+\kappa_{d-1}\rho^{d-1}V_1([0,\mu]\bigr) + \sum_{m=2}^d \rho^{d-m}\kappa_{d-m}V_m([0,\mu])\\
&=\kappa_d\rho^d+\kappa_{d-1}\rho^{d-1}V_1([0,\mu]\bigr),
\end{align*}
as $V_0\equiv 1$ and one can use the relation between intrinsic volumes and so-called mixed volumes (see \cite[Eq. (5.31)]{Schneider-book}) together with \cite[Theorem 5.1.8]{Schneider-book} to show that $V_m([0,\mu])=0$ for $m\in \{2,\ldots ,d\}$.  On the other hand, $\mathrm{Vol}_d\bigl([0,\mu]+\rho B^d\bigr)=\kappa_d\rho^d+\kappa_{d-1}\rho^{d-1}\lVert\mu\rVert$ and 
 the assertion follows.
 We finally show that $ \mathsf{p}( [0,\mu])=\frac{\mu}{2}$. We have
 \begin{align*}
  \mathsf{p}( [0,\mu]) &= \frac{1}{\kappa_d}\int_{\mathbb{S}^{d-1}}s_{[0,\mu]}(\theta)\theta \sigma (\mathrm{d}\theta) 
  =\frac{1}{\kappa_d}
  \int_{\{\theta \in \mathbb{S}^{d-1}: \langle \theta ,\mu\rangle >0\}}\langle \theta ,\mu \rangle \theta \sigma (\mathrm{d}\theta)\\
  &=\frac{1}{2\kappa_d}\int_{\mathbb{S}^{d-1}}\langle \theta ,\mu \rangle \theta \sigma (\mathrm{d}\theta).
 \end{align*}
 The claim is then a consequence of the fact that
 \begin{align*}
 \int_{\mathbb{S}^{d-1}}\langle \theta ,\mu \rangle \theta \sigma (\mathrm{d}\theta) = \kappa_d \mu
 \end{align*}
and the proof is finished.
\end{proof}

Even though we could apply the same reasoning as in \eqref{CONV-DRIFT-1} for $m \in \{2,\ldots ,d\}$ and obtain 
\begin{align*}
\frac{V_m(n)}{n^m}\xrightarrow[n\nearrow\infty]{\mathbb{P}\text{-a.s.}}V_m( [0,\mu]),
\end{align*} 
this would not provide accurate information as
$V_m( [0,\mu])=0$ for $m \in \{2,\ldots ,d\}$, see \cite[Theorem 5.1.8]{Schneider-book}.
We thus need to find an optimal scaling for the sequence $\{V_m(n)\}_{n\ge0}$ when $m\in \{2,\ldots ,d\}$ with different methods. For this 
we adapt approach developed in \cite{Wade-Xu_SPA} (see also \cite{McRedmond_survey}). We choose a standard basis of $\R^d$ according to the drift vector $\mu$ and then use scaling of linear order in the coordinate along the drift, while for the remaining coordinates we use sequence $\{b_n\}_{n\in\N}$. This results in a time-space L\'{e}vy process in the limit.

Let $\{ e_1,\dots, e_d\}$ be the standard orthonormal basis of $\R^d$
and let $\{\bar e_1,\dots,\bar e_d\}$ be an another orthonormal basis such that $\bar e_1=\mu/\lVert\mu\rVert$. Further, let $\phi:\R^d\to\R^{d}$ and $\phi^\bot:\R^d\to\R^{d-1}$ be two linear mappings given by
\begin{align*}
\phi(x)=(\langle x,\bar e_1\rangle,\dots, \langle x,\bar e_d\rangle)
\quad \mathrm{and}\quad
\phi^\bot(x)=(\langle x,\bar e_2\rangle,\dots, \langle x,\bar e_d\rangle).
\end{align*} 
Set $\bar {\mathbf{Y}}_i=\phi(\mathbf{Y}_i)$ and $\bar {\mathbf{Y}}^{\bot}_i=\phi^{\bot}(\mathbf{Y}_i)$ for  $i\in\N$. Clearly, $\{\bar {\mathbf{Y}}_i\}_{i\in\N}$ and $\{\bar {\mathbf{Y}}^\bot_i\}_{i\in\N}$ are sequences of independent and identically distributed $\R^d$, respectively, $\R^{d-1}$-valued random vectors. Also, $\mathbb{E}[\lVert \bar {\mathbf{Y}}_i\rVert^{\alpha-\varepsilon}]<\infty$ and $\mathbb{E}[\lVert \bar {\mathbf{Y}}^\bot_i\rVert^{\alpha-\varepsilon}]<\infty$ for all $\varepsilon\in(0,\alpha)$ and 
$$
\mathbb{E}\bigl[ \bar{Y}^{(k)}_i\bigr]=\langle \mu,\bar e_k\rangle=\begin{cases}
\lVert \mu\rVert, & k=1,\\
0, & k>1.
\end{cases}
$$ 
Further, let $\bar{\mathbf{S}}(n)=\sum_{i=1}^n\bar{\mathbf{Y}}_i$ and $\bar{\mathbf{S}}^\bot(n)=\sum_{i=1}^n\bar{\mathbf{Y}}^\bot_i$ be the corresponding random walks starting from the origin.
Observe that $\bar{\mathbf{S}}(n)=\phi(\mathbf{S}(n))$, $\bar{\mathbf{S}}^\bot(n)=\phi^\bot(\mathbf{S}(n))$ and $\{\bar{\mathbf{S}}^\bot(n)\}_{n\ge0}$ is a zero-drift random walk.
Assumption \eqref{DOM-ATT} and continuous mapping theorem yield
$$
\frac{\bar\bS(n)-n\lVert\mu\rVert e_1 }{b_n}=\phi\left(\frac{\mathbf{S}(n)-n\mu}{b_n}\right)\xrightarrow[n\nearrow\infty]{\mathcal{D}} \phi\bigl(\bX(1)\bigr)
$$ 
and  
\begin{equation*}
\frac{\bar\bS^\bot(n)}{b_n}=\phi^\bot\left(\frac{\mathbf{S}(n)-n\mu}{b_n}\right)\xrightarrow[n\nearrow\infty]{\mathcal{D}} \phi^\bot\bigl(\bX(1)\bigr).
\end{equation*}
The process $\bar{\bX}(t)=\phi(\bX(t))$, $t\ge0$, is necessary a $d$-dimensional $\alpha$-stable L\'evy process. From \cite[Theorem 2.1.5]{Samorodnitsky} it then follows that $\bar{\bX}^\bot(t)=\phi^\bot(\bX(t))$, $t\ge0$, is  a $(d-1)$-dimensional $\alpha$-stable L\'evy process.
For $n\in\N$ we define three linear mappings $\psi_n:\R^d\to\R^d$, $\psi_n^1:\R^d\to\R$ and $\psi_n^\bot:\R^d\to\R^{d-1}$ by
$$ \psi_n(x_1,\dots,x_d)=\phi\left(\frac{x_1}{n},\frac{x_2}{b_n},\dots,\frac{x_d}{b_n}\right),\qquad \psi_n^1(x_1,\dots,x_d)=\frac{\langle (x_1,\dots,x_d),\bar e_1\rangle}{n}$$ and
$$\psi_n^\bot(x_1,\dots,x_d)=\phi^\bot\left(x_1,\frac{x_2}{b_n},\dots,\frac{x_d}{b_n}\right).$$ For $A\subseteq\R^d$ we put $\psi_n(A)=\{\psi_n(x):x\in A\}$ and observe that if $A$ is compact/convex, then  
$\psi_n(A)$ is also compact/convex.
Let 
\begin{equation}\label{process tilde X}
\widetilde{\mathbf{X}}(t)=(\lVert\mu\rVert t,\bar{\bX}^\bot(t)),\quad t\geq 0.
\end{equation}
Clearly, $\{\widetilde{\mathbf{X}}(t)\}_{t\ge0}$ is a L\'evy process in $\R^d$. We now prove the following auxiliary result. We follow \cite[Lemma 6.6]{McRedmond_survey}.

\begin{proposition}\label{DRIFT}
Assume \eqref{DOM-ATT}. Let $\alpha >1$ and $\mu\neq 0$.
It holds 
$$
\psi_n\bigl(\mathrm{conv}\bigl(\{\bS(0),\ldots,\bS(n)\}\bigr)\bigr)\xRightarrow[n\nearrow\infty]{} \mathrm{conv}\,\widetilde\bX[0,1]
$$ 
in the sense of weak convergence in  $(\mathcal{K}_0^d,\mathrm{dist}_{\mathcal{H}})$.
	\end{proposition}
\begin{proof}
Due to linearity, $\psi_n(\mathrm{conv}(\{\bS(0),\ldots,\bS(n)\})=\mathrm{conv}(\psi_n(\{\bS(0),\ldots,\bS(n)\}))$. Let $\mathcal{C}^d$ be the collection of all compact sets in $\R^d$ and $\mathcal{C}^d_0=\{C\in \mathcal{C}^d:0\in C\}$. Since the mapping $\mathcal{C}^d \ni A\mapsto \mathrm{conv} A \in \mathcal{K}^d$ is continuous, it suffices to show that 
	$$
	\psi_n\bigl(\{\bS(0),\ldots,\bS(n)\}\bigr)\xRightarrow[n\nearrow\infty]{} \widetilde\bX[0,1]
	$$ 
	in the sense of weak convergence in  $(\mathcal{C}_0^d,\mathrm{dist}_{\mathcal{H}})$. This will follow if we prove that 
	\begin{equation}\label{FCLT-AUX}
	\left\{\psi_n\bigl(\bS(\lfloor nt\rfloor)\bigr)\right\}_{t\in[0,1]}\xrightarrow[n\nearrow\infty]{\mathcal{J}^{d}_1}\{\widetilde\bX(t)\}_{t\in[0,1]}.
	\end{equation}
	Indeed, since the mapping $f\mapsto f[0,1]$  is continuous from  $(\mathcal{D}^d_0[0,1],\mathrm{dist}_{\mathcal{J}_1})$ to $(\mathcal{C}_0^d,\mathrm{dist}_{\mathcal{H}})$ (see the proof of Lemma \ref{H-ctns}), if \eqref{FCLT-AUX} holds then continuous mapping theorem implies  
	$$
	\psi_n\bigl(\{\bS(0),\ldots,\bS(n)\}\bigr)=
	\cl \bigl\{\psi_n\bigl(\bS(\lfloor nt\rfloor)\bigr):t\in[0,1]\bigr\}\xRightarrow[n\nearrow\infty]{} \widetilde\bX[0,1],
	$$ 
	which proves the assertion.
		
To show \eqref{FCLT-AUX} we set $A_n=\{(\|\mu\|t,\psi_n^\bot(\bS(\lfloor nt\rfloor)))\}_{t\in[0,1]}$ and $B_n=\{\psi_n(\bS(\lfloor nt\rfloor))\}_{t\in[0,1]}$.  According to \cite[Theorem 3.1]{Billingsley-book},  \eqref{FCLT-AUX} will follow if we show that 
\begin{equation}\label{SLU}
	\mathrm{dist}_{\mathcal{J}_1}(A_n,B_n)\xrightarrow[n\nearrow\infty]{\mathbb{P}\text{-a.s.}}0 \qquad \text{and}\qquad A_n\xrightarrow[n\nearrow\infty]{\mathcal{J}^{d}_1}\{\widetilde\bX(t)\}_{t\in[0,1]}.
\end{equation}  
	To prove the first relation in \eqref{SLU} we proceed as follows. We have  
	\begin{align*}
	\mathrm{dist}_{\mathcal{J}_1}(A_n,B_n)&=\mathrm{dist}_{\mathcal{J}_1}\left(\bigl\{\bigl(\|\mu\|t,\psi_n^\bot\bigl(\bS(\lfloor nt\rfloor)\bigr)\bigr)\bigr\}_{t\in[0,1]},\bigl\{\bigr(\psi_n^1\bigr(\bS(\lfloor nt\rfloor)\bigr),\psi_n^\bot\bigl(\bS(\lfloor nt\rfloor)\bigr)\bigr)\bigr\}_{t\in[0,1]}\right)\\
	&\le\sup_{t\in[0,1]}\left|\psi_n^1\bigl(\bS(\lfloor nt\rfloor)\bigr)-\|\mu\|t\right|.
	\end{align*}
Observe  that $\psi_n^1(\bS(\lfloor nt\rfloor))=\bar{S}^{1}(\lfloor nt\rfloor)/n$, where $\{\bar{S}^{1}(n)\}_{n\ge0}$ is a one-dimensional random walk with drift $\lVert\mu\rVert$. Hence, according to \cite[Theorem 3.4]{McRedmond_survey},
$$
\sup_{t\in[0,1]}\left|\psi_n^1\bigl(\bS(\lfloor nt\rfloor)\bigr)-\|\mu\|t\right|\xrightarrow[n\nearrow\infty]{\mathbb{P}\text{-a.s.}}0.
$$
We next prove the second relation in  \eqref{SLU}. Note that $\psi_n^\bot(\bS(n))=\bar{\mathbf{S}}^\bot(n)/b_n$. By \eqref{DOM-ATT} and continuous mapping theorem we obtain (cf. \eqref{FCLT})
$$
\left\{\psi_n^\bot\bigl(\bS(\lfloor nt\rfloor)\bigr)\right\}_{t\in[0,1]}=\left\{\frac{\bar{\bS}^\bot(\lfloor nt\rfloor)}{b_n}\right\}_{t\in [0,1]}\xrightarrow[n\nearrow\infty]{\mathcal{J}^{d-1}_1}\{\bar{\bX}^\bot(t)\}_{t\in[0,1]}.
$$ 
This means that for any bounded and continuous $f:(\mathcal{D}^{d-1}_0[0,1],\mathrm{dist}_{\mathcal{J}_1})\to (\R,|\cdot|)$ it holds that 
$$
\lim_{n\to\infty}\mathbb{E}\left[f\left(\left\{\psi_n^\bot\bigl(\bS(\lfloor nt\rfloor)\bigr)\right\}_{t\in[0,1]}\right)\right]=\mathbb{E}\left[f\left(\{\bar{\bX}^\bot(t)\}_{t\in[0,1]}\right)\right].
$$ 
For a given $g:(\mathcal{D}^{d}_0[0,1],\mathrm{dist}_{\mathcal{J}_1})\to (\R,|\cdot|)$ which is continuous and bounded we define $$f(\cdot)=g\left(\bigl\{\lVert\mu \rVert t,\cdot\bigr\}_{t\in[0,1]}\right).$$
	Clearly, $f:(\mathcal{D}^{d-1}_0[0,1],\mathrm{dist}_{\mathcal{J}_1})\to (\R,|\cdot|)$ is also continuous and bounded. Hence, 
	\begin{align*}
	\lim_{n\to\infty}\mathbb{E}\left[g\left(\bigl\{\bigl(\|\mu\|t,\psi_n^\bot\bigl(\bS(\lfloor nt\rfloor)\bigr)\bigr)\bigr\}_{t\in[0,1]}\right)\right]
	&=\lim_{n\to\infty}\mathbb{E}\left[f\left(\left\{\psi_n^\bot\bigl(\bS(\lfloor nt\rfloor)\bigr)\right\}_{t\in[0,1]}\right)\right]\\
	&=\mathbb{E}\left[f\left(\{\bar{\bX}^\bot(t)\}_{t\in[0,1]}\right)\right]\\
	&=\mathbb{E}\left[g\left(\bigl\{\bigl(\lVert\mu\rVert t,\bar{\bX}^\bot(t)\bigr)\bigr\}_{t\in[0,1]}\right)\right]\\
	&=\mathbb{E}\left[g\left(\bigl\{\widetilde{\bX}(t)\bigr\}_{t\in[0,1]}\right)\right],
	\end{align*} which completes the proof.
	\end{proof}

Let $T=(\langle \bar e_k, e_l\rangle)_{k,l=1}^d$ and $D_n=\mathrm{diag}(1/n,1/b_n\dots,1/b_n)$, and observe that $T$ is orthogonal,  $\phi(x)=Tx$ and $\psi_n(x)=TD_nx$. The following result provides information on the convergence of the $m$-th, $m\in \{2,\ldots ,d\}$, intrinsic volume of $\mathrm{conv}\{\bS(0),\ldots,\bS(n)\}$ and it should be compared with \cite[Theorem 6.13]{McRedmond_survey} and \cite[Corollary 2.8]{Wade-Xu_SPA}.

\begin{theorem}\label{CONV-DRIFT-2}
Assume \eqref{DOM-ATT}. Let $\alpha >1$ and $\mu\neq 0$.
	For $m\in \{1,\ldots ,d\}$ it holds 
$$
V_m\left(D_n\bigl(\mathrm{conv}\{\bS(0),\ldots,\bS(n)\}\bigr)\right)\xrightarrow[n\nearrow\infty]{\mathcal{D}} V_m\bigl(\mathrm{conv}\,\widetilde\bX[0,1]\bigr).
$$ 
In particular, 
$$
\frac{V_d(n)}{nb_n^{d-1}}\xrightarrow[n\nearrow\infty]{\mathcal{D}} V_d\bigl(\mathrm{conv}\,\widetilde\bX[0,1]\bigr).
$$ 
	\end{theorem}
\begin{proof}
By Proposition \ref{DRIFT} and continuous mapping theorem it follows that 
	$$
	V_m\left(\psi_n\bigl(\mathrm{conv}\{\bS(0),\ldots,\bS(n)\})\bigr)\right)\xrightarrow[n\nearrow\infty]{\mathcal{D}} V_m\bigl(\mathrm{conv}\,\widetilde\bX[0,1]\bigr).
	$$ 
Since
	$$
	V_m\left(\psi_n\bigl(\mathrm{conv}\{\bS(0),\ldots,\bS(n)\}\bigr)\right)=V_m\left(TD_n\bigl(\mathrm{conv}\{\bS(0),\ldots,\bS(n)\}\bigr)\right)
	$$ 
and intrinsic volumes are invariant under orthogonal transformations (see \cite[Chapter III]{Schneider}), we infer the first assertion. 
The second formula follows from the fact that 
$$
V_d\left(D_n\bigl(\mathrm{conv}\{\bS(0),\ldots,\bS(n)\}\bigr)\right)=\det(D_n)V_d(n)=\frac{V_d(n)}{nb_n^{d-1}}
$$
and the proof is finished.
\end{proof}

\begin{remark}\label{REM}
	 If we assume that $\bbE[\lVert \bY_1\rVert^2]<\infty$ and define $\Sigma_{\bot}=\bbE[(\bar{\bY}^\bot_1)(\bar{\bY}^\bot_1)^t]$, then $b_n=\sqrt{n}$ (see \cite[Theorem 2.6.6]{Ibragimov-Linnik}) and  
	$$
	\{\bar{\bX}^\bot(t)\}_{t\ge0}\stackrel{\mathcal{D}}{=}\{\Sigma_\bot^{1/2}\mathbf{B}(t)\}_{t\ge0},
	$$ 
	where $\{\mathbf{B}(t)\}_{t\ge0}$ is a standard $(d-1)$-dimensional Brownian motion. Further, let $$\Sigma_{\mu}=\begin{pmatrix}
	\lVert\mu\rVert^2 & 0 \\
	0 & \Sigma_{\bot}
	\end{pmatrix}.$$ Then, 
	$$\bigl\{\widetilde{\bX}(t)\bigr\}_{t\ge0}\stackrel{\mathcal{D}}{=}\big\{\Sigma_{\mu}^{1/2}\widetilde{\mathbf{B}}(t)\big\}_{t\ge0}
	$$
	and 
	$$ 
	V_d\bigl(\mathrm{conv}\, \widetilde\bX[0,1]\bigr)
	=
	\sqrt{\det (\Sigma_{\mu})}\, V_d\bigl(\mathrm{conv}\, \widetilde{\mathbf{B}}[0,1]\bigr)
	=
	\lVert\mu\rVert\sqrt{\det (\Sigma_{\bot})}\, V_d\bigl(\mathrm{conv}\, \widetilde{\mathbf{B}}[0,1]\bigr),
	$$ where $\widetilde{\mathbf{B}}(t)=(t,\mathbf{B}(t)),$ $t\ge0$.
		In view of Theorem \ref{CONV-DRIFT-2} we obtain
	$$
	V_m\left(\psi_n\bigl(\mathrm{conv}\{\bS(0),\ldots,\bS(n)\})\bigr)\right)\xrightarrow[n\nearrow\infty]{\mathcal{D}} 
	V_m \bigl(\Sigma_\mu^{1/2}\mathrm{conv}\, \tilde{\mathbf{B}}[0,1] \bigr)
	$$ and 
	$$V_d\left(\psi_n\bigl(\mathrm{conv}\{\bS(0),\ldots,\bS(n)\})\bigr)\right)\xrightarrow[n\nearrow\infty]{\mathcal{D}}  
		\lVert\mu\rVert\sqrt{\det (\Sigma_{\bot})}\, V_d\bigl(\mathrm{conv}\, \widetilde{\mathbf{B}}[0,1]\bigr).$$
\end{remark}

\section{Convergence of means}\label{sec:Means}
In this section, we study convergence of expected intrinsic volumes of $\conv \SRange$.
We assume that $\alpha>1$ which implies
 $\mathbb{E}[\lVert\bY_1\rVert^{\alpha-\varepsilon}]<\infty$ for every $0<\varepsilon<\alpha.$ 
As before, we shall distinguish between two cases: $\mu = \mathbb{E}[\bY_1]=0$, or $\mu \neq 0$. 

\subsection{Zero-drift case}
In this paragraph, we assume that $\mu=0$. The following result concerns the sequences $\{V_1(n)\}_{n\ge0}$ and $\{P(n)\}_{n\ge0}$. It should be compared with \cite[Proposition 3.1]{Wade-Xu_SPA}, where the expected perimeter of the convex hull of planar random walks was studied. 
\begin{theorem}\label{PER}
Assume \eqref{DOM-ATT}. Let $\alpha >1$ and $\mu =0$.
It holds
\begin{equation}\label{Exp_V_1_Limit}
\lim_{n\to \infty}\frac{\bbE [V_1(n)]}{b_n}
=
\bbE \bigl[V_1 \bigl(\mathrm{conv}\, \bX[0,1] \bigr)\bigr]
\end{equation}
and
\begin{equation}\label{P_1}
\lim_{n\to \infty}\frac{\bbE \bigl[\|P(n)\|\bigr]}{b_n}
=
	\bbE \bigl[\|\mathsf{p} \bigl(\mathrm{conv}\, \bX[0,1] \bigr)\|\bigr].
\end{equation}
\end{theorem}

\begin{proof}
In view of Proposition \ref{FCLT-HULL}, it suffices to show that the sequences $\{\|P(n)\|/b_n\}_{n\geq 1}$ and $\{V_1(n)/b_n\}_{n\geq 1}$ are uniformly integrable (see \cite[Lemma 3.11]{Kallenberg}). To show \eqref{P_1} we proceed as follows. It is evident from the definition that  $
V_1(n)\leq (d\kappa_d/\kappa_{d-1}) \max_{1\leq k\leq n}\Vert \bS(k)\Vert
$ and $\|P(n)\|\leq (\varpi_d/\kappa_d) \max_{1\leq k\leq n}\Vert \bS(k)\Vert$. We fix $\varepsilon>0$ such that $\alpha -\varepsilon >1$. Doob's maximal inequality yields
\begin{align*}
\bbE \bigl[\bigl(V_1(n)/b_n\bigr)^{\alpha -\varepsilon}\bigr]\leq \left(\frac{d\kappa_d (\alpha -\varepsilon)}{\kappa_{d-1}(\alpha -\varepsilon -1)}\right)^{\alpha -\varepsilon}\bbE \bigl[\Vert S(n)/b_n\Vert ^{\alpha -\varepsilon}\bigr].
\end{align*} and
\begin{align*}
\bbE \bigl[\bigl(\|P(n)\|/b_n\bigr)^{\alpha -\varepsilon}\bigr]\leq \left(\frac{\varpi_d (\alpha -\varepsilon)}{\kappa_d(\alpha -\varepsilon -1)}\right)^{\alpha -\varepsilon}\bbE \bigl[\Vert S(n)/b_n\Vert ^{\alpha -\varepsilon}\bigr].
\end{align*}
Hence it is enough to prove that the sequence $\{\Vert S(n)\Vert /b_n\}_{n\geq 1} $  is uniformly bounded in $\mathrm{L}^{\alpha -\varepsilon}$.  Since we have 
$$\bbE \bigl[\Vert \bS(n)/b_n\Vert ^{\alpha -\varepsilon}\bigr] \leq \sum_{k=1}^d \bbE \bigl[\vert S^{(k)}(n)/b_n \vert ^{\alpha -\varepsilon}\bigr],$$
we only need to show that the moments of order $\alpha-\varepsilon$ of the coordinates are uniformly bounded.
 This  follows from \cite[Lemma 5.2.2]{Ibragimov-Linnik}, since the step distributions of $\{S^{(k)}(n)\}_{n\ge0}$, $k=1,\dots,d$, belong to the domain of attraction of a one-dimensional $\alpha$-stable   law (see \cite[Theorem 2.1.2]{Samorodnitsky}). 
\end{proof}

\begin{remark}
We point out that $\bbE \left[ \left( V_m(\conv \, \bX[0,1])\right)^{p}\right] < \infty$ for all $p\in [0,\alpha)$ and $m\in \{1,\ldots ,d\}$ in view of \cite[Theorem 1.1]{Molchanov-Wespi}. 
\end{remark}
We next present the corresponding result for the remaining mean intrinsic volumes.
Here we pose an extra (the so-called general position) assumption on the walk $\{\bS(n)\}_{n\ge0}$, namely we require that it does not stay in any affine hyperplane of $\R^d$ with probability one. This in turn implies that for any choice of time indices $1\leq j_1<\ldots <j_d$, the random vectors $\bS(j_1),\ldots ,\bS(j_d)$ must be almost surely linearly independent, see \cite{Nielsen-Baxter_Area} and \cite[Proposition 2.5]{Kabluchko-Vys-Zap_GAFA}. Under this condition, evidently, the distributions of all coordinates are  continuous and thus lattice random walks are excluded.
Our result can be viewed as a generalization of \cite[p.\ 325]{Nielsen-Baxter_Area} and \cite[Proposition 3.3]{Wade-Xu_SPA}  which concerned the asymptotic behaviour of the area of the convex hull of planar random walks.
\begin{theorem}\label{Vm} 
Assume \eqref{DOM-ATT}. Let $\alpha >1$ and $\mu =0$.
Suppose that $\Prob(\bY_1\in \mathbb{h})=0$ for any affine hyperplane $\mathbb{h}\subset\R^d$.
Then, for each $m\in\{1,\dots,d\}$,
\begin{align}\label{EXP_Intr_Vol_Limit}
\lim_{n\to \infty}
	\frac{\bbE [V_m(n)]}{b^m_n}
	=
		\frac{\alpha\Gamma(1/\alpha)^m}{m\Gamma(m/\alpha)} \cdot
	\frac{\bbE\left[\sqrt{\det \left(\langle\bX^{(k)}(1),\bX^{(l)}(1)\rangle\right)_{k,l=1}^{m}}\right]}{m!},
\end{align}
where $\{\bX^{(k)}(t)\}_{t\ge0}$, $k=1,\dots,m$, are independent $\alpha$-stable L\'evy processes with the same law as $\{\bX(t)\}_{t\ge0}$.
\end{theorem}
\begin{proof} 
%
%
According to \cite[Corollary 3]{Vysotsky} we have
\begin{align}\label{Vysot}
\bbE[V_m(n)]=\frac{1}{m!}\sum_{\substack{j_1+\cdots+j_m\le n \\ j_1,\dots,j_m\in\N}}\frac{\bbE\left[\sqrt{\det \left(\langle\bS^{(k)}(j_k),\bS^{(l)}(j_l)\rangle\right)_{k,l=1}^{m}}\right]}{j_1\cdots j_m},
\end{align}
where 
$\{\bS^{(k)}(n)\}_{n\ge0}$, $k=1,\dots,m$, are independent random walks with the same law as $\{\bS(n)\}_{n\ge0}$. Note that the determinant in \eqref{Vysot} is a special case of the Gram determinant and it is always non-negative. In view of the general position assumption it is actually positive.  
We use notation
\begin{align*}
\Delta^{\bS}(j_1,\ldots ,j_m) = \sqrt{\det \left(\langle\bS^{(k)}(j_k),\bS^{(l)}(j_l)\rangle\right)_{k,l=1}^{m}}
\end{align*}
and
\begin{align*}
\Delta_m^{\bX} = \sqrt{\det \left(\langle\bX^{(k)}(1),\bX^{(l)}(1)\rangle\right)_{k,l=1}^{m}}.
\end{align*}
Observe that   $\R^{d\times m}\ni(x_1,\dots,x_m)\mapsto \sqrt{\det(\langle  x_k,x_l\rangle)_{k,l=1}^m}$ is continuous and 
$$
\sqrt{\det\left(\langle   a_kx_k,a_lx_l\rangle\right)_{k,l=1}^m}=a_1\cdots a_m\sqrt{\det\left(\langle   x_k,x_l\rangle\right)_{k,l=1}^m},
$$ 
for $a_1,\dots,a_m\in[0,\infty)$. 
Thus, continuous mapping theorem implies
$$
 \frac{\Delta^{\bS}(j_1,\ldots ,j_m)}{b_{j_1}\cdots b_{j_m}}\xrightarrow[]{\mathcal{D}}
\Delta^{\bX}_m,\qquad \mathrm{as}\quad j_1,\dots,j_m\to \infty.
$$ 
Next, for any $\varepsilon>0$ such that $\alpha-\varepsilon>1$, it holds  
$$
\bbE\left[\left(\frac{\Delta^{\bS}(j_1,\ldots ,j_m)}{b_{j_1}\cdots b_{j_m}}\right)^{\alpha-\epsilon}\right]\le \bbE\left[\prod_{k=1}^m\left\|\frac{\bS^{(k)}(j_k)}{b_{j_k}}\right\|^{\alpha-\varepsilon}\right]=\prod_{k=1}^m\bbE\left[\left\|\frac{\bS^{(k)}(j_k)}{b_{j_k}}\right\|^{\alpha-\varepsilon}\right],
$$
where in the first inequality we used Hadamard's inequality  for the Gram determinant, that is
$$
\sqrt{\det\left(\langle   x_k,x_l\rangle\right)_{k,l=1}^m}\le \lVert x_1\rVert\cdots\lVert x_m\rVert.
$$ 
Hence,
by using the same argument as in Proposition \ref{PER} we infer that 
\begin{equation}\label{eq:BDD}
\sup_{j_1,\dots,j_m\ge1}
\bbE\left[\left(\frac{\Delta^{\bS}(j_1,\ldots ,j_m)}{b_{j_1}\cdots b_{j_m}}\right)^{\alpha-\epsilon}\right]
<\infty.
\end{equation}
It follows 
$$\lim_{j_1,\dots,j_m\to\infty}\frac{\bbE\left[\Delta^{\bS}(j_1,\ldots ,j_m) \right]}{b_{j_1}\cdots b_{j_m}}=\bbE\left[\Delta^{\bX}_m\right].
$$
We  denote
\begin{align}\label{error_term}
\varepsilon_{j_1,\dots,j_m}=\frac{\bbE\left[\Delta^{\bS}(j_1,\ldots ,j_m) \right]}{b_{j_1}\cdots b_{j_m}\bbE\left[\Delta^{\bX}_m\right]}-1.
\end{align}
Clearly $\lim_{j_1,\dots,j_m\to\infty}\varepsilon_{j_1,\dots,j_m}=0$ and
 proceeding similarly as in \eqref{eq:BDD} we can show that
\begin{align}\label{eq:BDD1}
\sup_{j_1,\dots,j_m\ge1}|\varepsilon_{j_1,\dots,j_m}|<\infty.
\end{align}
Combining \eqref{Vysot} and \eqref{error_term} we obtain
\begin{equation}\label{CONV1} 
\frac{\bbE[V_m(n)]}{b_n^m}=\frac{\bbE\left[\Delta^{\bX}_m\right]}{m!}\frac{1}{b_n^m}\sum_{\substack{j_1+\cdots+j_m\le n \\ j_1,\dots,j_m\in\N}}\frac{b_{j_1}\cdots b_{j_m}}{j_1\cdots j_m}+\frac{\bbE\left[\Delta^{\bX}_m\right]}{m!}\frac{1}{b^m_n}\sum_{\substack{j_1+\cdots+j_m\le n \\ j_1,\dots,j_m\in\N}}\frac{b_{j_1}\cdots b_{j_m}}{j_1\cdots j_m}\varepsilon_{j_1,\dots,j_m}.
\end{equation}
We start with the first term in \eqref{CONV1}. Let $a_n=b_n/n$.  We then have 
$$
\lim_{n\to\infty}\frac{1}{b_n^m}\sum_{\substack{j_1+\cdots+j_m\le n \\ j_1,\dots,j_m\in\N}}\frac{b_{j_1}\cdots b_{j_m}}{j_1\cdots j_m}=\lim_{n\to\infty}\frac{1}{b_n^m}\sum_{\substack{j_1+\cdots+j_m\le n \\ j_1,\dots,j_m\in\N}}a_{j_1}\cdots a_{j_m}.
$$
Recall that for two sequences $\{x_n\}_{n\ge0},\{y_n\}_{n\ge0}\subset\R$ their convolution is defined as 
\begin{align*}
(x\ast y)_n=\sum_{m=0}^{n}x_m y_{n-m}.
\end{align*} 
For $n\ge0$ and $k\ge2$ we use notation  $x^{\ast 1}_n=x_n$  and, inductively,  $x^{\ast k}_n=(x\ast x^{\ast(k-1)})_n.$ 
We next observe that 
  \begin{equation}\label{CONV_0}
\lim_{n\to\infty}\frac{1}{b_n^m}\sum_{\substack{j_1+\cdots+j_m\le n \\ j_1,\dots,j_m\in\N}}a_{j_1}\cdots a_{j_m}=\lim_{n\to\infty}\frac{1}{b_n^m}\sum_{k=m}^na^{\ast m}_k.
\end{equation}
Indeed, without  loss of generality we can assume that $\{b_n\}_{n\in\N}$ is strictly monotone (see \cite[Theorem 1.5.3]{BGT_book}) and we clearly have
\begin{align*}
\sum_{\substack{j_1+\cdots+j_m\le n+1 \\ j_1,\dots,j_m\in\N}}a_{j_1}\cdots a_{j_m}
-
\sum_{\substack{j_1+\cdots+j_m\le n \\ j_1,\dots,j_m\in\N}}a_{j_1}\cdots a_{j_m}
&=\sum_{\substack{j_1+\cdots+j_m= n+1 \\ j_1,\dots,j_m\in\N}}a_{j_1}\cdots a_{j_m}\\
&=
a_{n+1}^{*m}
=
\sum_{k=m}^{n+1}a^{\ast m}_k
-
\sum_{k=m}^na^{\ast m}_k.
\end{align*}
Hence, according to Stolz-Ces\`{a}ro theorem \cite[Theorem 1.22]{Marian} we infer that the limits in \eqref{CONV_0} coincide (if they exist). 
We next compute the limit in \eqref{CONV_0}. We claim that
\begin{equation}\label{CONV}
\lim_{n\to\infty}\frac{1}{b_n^m}\sum_{k=m}^na^{\ast m}_k=\frac{\alpha\Gamma(1/\alpha)^{m}}{m\Gamma(m/\alpha)}.
\end{equation}
	We prove this through induction over $m$. 
	Since $b_n=n^{1/\alpha}\ell(n)$, we have $a_n=n^{1/\alpha-1}\ell(n)$. According to \cite[Lemma 2.4]{Nagaev}, for $m=1$, it  holds
	\begin{equation*}\label{eq:CONV1}\lim_{n\to\infty}\frac{1}{b_n}\sum_{k=1}^na_k=\lim_{n\to\infty}\frac{1}{n^{1/\alpha}\ell(n)}\sum_{k=1}^na_k=\alpha.\end{equation*}
	Suppose  that \eqref{CONV} holds for $1,\dots,m-1$. In view of \cite[Lemma 2.1]{Omey}, we obtain 
	\begin{align*}\lim_{n\to\infty}\frac{1}{b_n^m}\sum_{k=m}^na^{\ast m}_k&=\lim_{n\to\infty}\frac{1}{b_n^m}\sum_{k=m}^n\bigl(a\ast a^{\ast (m-1)}\bigr)_k
	\\&=\frac{\Gamma(1+1/\alpha)\Gamma(1+(m-1)/\alpha)}{\Gamma(1+m/\alpha)}\frac{\alpha^2\Gamma(1/\alpha)^{m-1}}{(m-1)\Gamma((m-1)/\alpha)}\\
	&=\frac{\alpha\Gamma(1/\alpha)^{m}}{m\Gamma(m/\alpha)},\end{align*}
	as desired.

	Finally,  we show that the second term in \eqref{CONV1} converges to zero. 
Fix $\varepsilon\in(0,1)$ and let $j_0\in\N$ be such that $|\varepsilon_{j_1,\dots,j_m}|<\varepsilon$ for all $j_1,\dots,j_m\ge j_0$.
We have 
\begin{align*}
&\frac{1}{b_n^m}\sum_{\substack{j_1+\cdots+j_m\le n \\ j_1,\dots,j_m\in\N}}\frac{b_{j_1}\cdots b_{j_m}}{j_1\cdots j_m}\varepsilon_{j_1,\dots,j_m}\\
&=\frac{1}{b^m_n}\sum_{\substack{j_1+\cdots+j_m\le n \\ 1\le j_1,\dots,j_m\le j_0-1}}\frac{b_{j_1}\cdots b_{j_m}}{j_1\cdots j_m}\varepsilon_{j_1,\dots,j_m}+\frac{1}{b^m_n}\sum_{\substack{j_1+\cdots+j_m\le n \\ j_k\le j_0-1\ \text{and}\ j_l\ge j_0\ \text{for some}\ k,l}}\frac{b_{j_1}\cdots b_{j_m}}{j_1\cdots j_m}\varepsilon_{j_1,\dots,j_m} \\ &\ \ \ \ \ +\frac{1}{b^m_n}\sum_{\substack{j_1+\cdots+j_m\le n \\  j_1,\dots,j_m\ge j_0}}\frac{b_{j_1}\cdots b_{j_m}}{j_1\cdots j_m}\varepsilon_{j_1,\dots,j_m}.
\end{align*}
The first term clearly converges to zero. By the choice of $j_0$ and  \eqref{CONV} it holds
 $$\limsup_{n\to\infty}\frac{1}{b^m_n}\sum_{\substack{j_1+\cdots+j_m\le n \\  j_1,\dots,j_m\ge j_0}}\frac{b_{j_1}\cdots b_{j_m}}{j_1\cdots j_m}|\varepsilon_{j_1,\dots,j_m}|\le \frac{\alpha\Gamma(1/\alpha)^m}{m\Gamma(m/\alpha)}\varepsilon.$$ 
Finally, from \eqref{eq:BDD1} and  \eqref{CONV} it follows that
\begin{align*}&\limsup_{n\to\infty}\frac{1}{b^m_n}\sum_{\substack{j_1+\cdots+j_m\le n \\ j_k\le j_0-1\ \text{and}\ j_l\ge j_0\ \text{for some}\ k,l}}\frac{b_{j_1}\cdots b_{j_m}}{j_1\cdots j_m}|\varepsilon_{j_1,\dots,j_m}|\\
&\le \limsup_{n\to\infty}\frac{m(m-1)2^{m-1}}{b^m_n}\sup_{j_1,\dots,j_m\ge1}|\varepsilon_{j_1,\dots,j_m}|\sum_{\substack{j_1+\cdots+j_m\le n \\ j_1\le j_0-1\ \text{and}\ j_k\ge j_0\ \text{for all}\ k\neq 1}}\frac{b_{j_1}\cdots b_{j_m}}{j_1\cdots j_m}\\
&\le m(m-1)2^{m-1}\sup_{j_1,\dots,j_m\ge1}|\varepsilon_{j_1,\dots,j_m}|\sum_{j_1=1}^{j_0-1}\frac{b_{j_1}}{j_1}\limsup_{n\to\infty}\frac{1}{b^m_n}\sum_{\substack{j_2+\cdots+j_m\le n \\ j_2,\dots,j_m\in\N}}\frac{b_{j_2}\cdots b_{j_m}}{j_2\cdots j_m}\\
&=m(m-1)2^{m-1}\sup_{j_1,\dots,j_m\ge1}|\varepsilon_{j_1,\dots,j_m}| \frac{\alpha\Gamma(1/\alpha)^{(m-1)}}{(m-1)\Gamma((m-1)/\alpha)}\sum_{j_1=1}^{j_0-1}\frac{b_{j_1}}{j_1}\lim_{n\to\infty}\frac{1}{b_n} \\
&=0,
\end{align*}
and the proof is finished.
\end{proof}

For $m=1$ we are allowed to abandon the assumption that $\Prob(\bY_1\in \mathbb{h})=0$ for any affine hyperplane $\mathbb{h}\subset\R^d$ and this is justified by Theorem \ref{PER}. Combining \eqref{Exp_V_1_Limit} and \eqref{EXP_Intr_Vol_Limit} enables us to conclude the following interesting result.
\begin{corollary}
Assume \eqref{DOM-ATT}. Let $\alpha >1$ and $\mu =0$.
It holds
\begin{align}\label{V_1_Stable}
\bbE \bigl[V_1 \bigl(\mathrm{conv}\,\bX[0,1] \bigr)\bigr]=\alpha\,\bbE\bigl[\lVert\bX(1)\rVert\bigr].
\end{align}
\end{corollary}

\begin{remark} Formula \eqref{V_1_Stable} is valid for any (even non-symmetric) $\alpha$-stable L\'evy process\linebreak $\{\bX(t)\}_{t\ge0}$ in $\bbR^d$ with $\mu=0$. In particular, if $\{\bX(t)\}_{t\ge0}$ is standard Brownian motion then $\Vert \bX (1)\Vert^2$ has chi-squared distribution $\chi^2(d)$ with $d$ degrees of freedom and this enables us to recover the following known relation (see \cite[Corollary 1.4]{Kampf-Molchanov})
		\begin{equation}\label{BM}
		\bbE \left[ V_1\left(\conv \, \bX[0,1]\right)\right] = \frac{2\sqrt{2}\Gamma ((d+1)/2)}{\Gamma (d/2)}
		\end{equation} 
		which is a generalization of the famous formula for the perimeter of the convex hull of planar Brownian motion, see \cite{Eldan} and \cite{Takacs}. 
		We could compute the first absolute moment of $\bX(1)$ also for rotationally invariant $\alpha$-stable L\'evy processes (see e.g.\ \cite[Eq.\ (7.5.9)]{Zolotarev}) but the corresponding formula for the first mean intrinsic volume for such processes is included in \eqref{rot-inv}.
\end{remark}

It turns out that when $\{\bX(t)\}_{t\ge0}$ is a symmetric $\alpha$-stable L\'evy processes we can obtain a more explicit form of the limit in \eqref{EXP_Intr_Vol_Limit}.  The characteristic function of  $\{\bX(t)\}_{t\ge0}$   is then given by 
$$
\bbE \left[ \exp \bigl(i\langle \xi , \bX(1)\rangle \bigr) \right] = \exp \left(-\int_{\mathbb{S}^{d-1}}|\langle \theta,\xi \rangle|^{\alpha}\varsigma(\mathrm{d}\theta)\right),\qquad \xi\in \mathbb{R}^d,
$$  
where $\varsigma(\mathrm{d}\theta)$ denotes the  corresponding (finite and symmetric) spectral measure,  see \cite[Theorem 2.4.3]{Samorodnitsky}. This together with Minkowski inequality and \cite[Theorem 1.7.1]{Schneider-book} (here we use the fact that $\alpha>1$) implies that  there is a unique  $K\in\mathcal{K}^d$ such that 
\begin{align*}
	\bbE \left[ \exp \bigl(i\langle \xi , \bX(1)\rangle \bigr) \right] = \exp \left(-s_K^\alpha (\xi)\right),\qquad \xi\in \mathbb{R}^d .
\end{align*} 
Here, $s_K(x)$ stands for the support function of the set $K$ and
the set $K$ is the so-called associated zonoid of $\bX (1)$, see \cite{Kampf-Molchanov} and \cite{Molchanov_JMA}. Using this fact, in
 \cite[Theorem 2.3]{Molchanov-Wespi} it is further shown that 
\begin{align}\label{Intr_Vol_for X}
	\bbE \left[ V_m\bigl(\conv\, \bX [0,1]\bigr) \right] = \frac{\alpha \Gamma(1/\alpha)^m \Gamma\big(1-1/\alpha\big)^m}{m\pi^m\Gamma(m/\alpha)}V_m(K).
\end{align}
In particular, if $\{\bX(t)\}_{t\ge0}$ is a standard Brownian motion then $K=(\sqrt{2}/2)B^d$ and 
\begin{align}\label{BM-f}
\bbE \left[ V_m\bigl(\conv\, \bX [0,1]\bigr) \right] =\binom{d}{m}\left(\frac{\pi}{2}\right)^{m/2}\frac{\Gamma \big( (d-m)/2+1\big)}{\Gamma \big( m/2+1\big) \Gamma \big( d/2+1\big)},
\end{align}
see \cite[Example 3.2]{Molchanov_JMA}, \cite[Corollary 1.2]{Eldan}, \cite[Example 2.5]{Molchanov-Wespi}, or \cite[Eq.\ (16)]{Kabluchko_Zapor-TAMS}. For $m=1$ this is formula \eqref{BM}.
If $\{\bX(t)\}_{t\ge0}$ is a rotationally invariant $\alpha$-stable L\'evy process with $\bbE [ \exp (i\langle \xi , \bX(1)\rangle ) ]=\exp(-\gamma|\xi|^\alpha)$ for some $\gamma>0$, then $K=\gamma^{1/\alpha}B^d$
and \begin{equation}\label{rot-inv}\bbE \left[ V_m\bigl(\conv\, \bX [0,1]\bigr) \right] =\binom{d}{m}\frac{\kappa_d}{\kappa_{d-m}}
\frac{\alpha\Gamma(1/\alpha)^m \Gamma \big(1-1/\alpha\big)^m}{m\pi^m\Gamma(m/\alpha)}
\gamma^{m/\alpha},\end{equation}
see \cite[Example 2.6]{Molchanov-Wespi}.

We  summarize the above discussion in the following corollary. 

\begin{corollary}\label{Cor_Formulas}
	Assume \eqref{DOM-ATT} with $\{\bX(t)\}_{t\ge0}$ being a symmetric $\alpha$-stable L\'evy process with  $\alpha >1$ and suppose that $\mu =0$.
	For $m\in \{ 2,\ldots ,d\}$ we assume additionally that $\Prob(\bY_1\in \mathbb{h})=0$ for any affine hyperplane $\mathbb{h}\subset\R^d$.  For all $m\in \{1,\ldots ,d\}$ it then holds
	\begin{align*}
	\lim_{n\to \infty}
	\frac{\bbE [V_m(n)]}{b^m_n}
	=
	 \bbE\left[V_m \bigl(\conv \,\bX [0,1]\bigr)\right]. 
	\end{align*}
In particular, the limit in \eqref{EXP_Intr_Vol_Limit} is given by \eqref{Intr_Vol_for X}. 
If, moreover, $\{\bX(t)\}_{t\ge0}$ is rotationally invariant with characteristic function $\exp(-\gamma|\xi|^\alpha)$ for some $\gamma>0$, then the limit is given by \eqref{rot-inv}.
\end{corollary}

\begin{proof}
The proof relies on some facts from the theory of random compact sets.
 It was proved in \cite[Corollary 2.2]{Molchanov-Wespi}  that 
$$
\frac{\bbE\left[\sqrt{\det \left(\langle\bX^{(k)}(1),\bX^{(l)}(1)\rangle\right)_{k,l=1}^{m}}\right]}{m!}
=
V_m\bigl(\bbE_A[0,\bX(1)]\bigr) ,
 $$
 where $V_m(\bbE_A[0,\bX(1)])$
 denotes the $m$-th intrinsic volume of the so-called Aumann expectation of the (random) segment $[0,\bX(1)]$, see \cite{Molchanov_book}. Thus,
\begin{equation}\label{Aum}
 \lim_{n\to \infty}
 \frac{\bbE [V_m(n)]}{b^m_n}
 =
 \frac{\alpha\Gamma(1/\alpha)^m}{m\Gamma(m/\alpha)} V_m\bigl(\bbE_A[0,\bX(1)]\bigr).\end{equation}
According to  \cite[Theorem 6.16]{Molchanov_JMA} it holds $$V_m\bigl(\bbE_A[0,\bX(1)]\bigr)=\frac{\Gamma(1-1/\alpha)^m}{\pi^m}V_m(K),$$ where $K$ is the associated zonoid of $\bX (1)$. Combining this with \eqref{Intr_Vol_for X} and \eqref{Aum} finishes the proof.
\end{proof}

\begin{remark}
	Suppose that  $\bbE[\lVert \bY_1\rVert^2]<\infty$ and recall Remark \ref{REM1}. Then, in view of Corollary \ref{Cor_Formulas} and \eqref{BM-f}, we  have
	$$ \lim_{n\to \infty}
	 \frac{\bbE [V_d(n)]}{n^{d/2}}
	 =\frac{\pi^{d/2}}{2^{(d-4)/2}d^2 \Gamma ( d/2)^2} \sqrt{\det(\Sigma)} 
	 .$$
	\end{remark}

\subsection{Non-zero drift case}
In this paragraph, we investigate the case $\mu\neq 0$. We start with a result for sequences $\{V_1(n)\}_{n\ge0}$ and $\{P(n)\}_{n\ge0}$. It follows from \cite[Theorem 1.1]{McRedmond-Wade_EJP} that for all planar random walks with drift $\mu\neq 0$ it holds $\lim_{n\to \infty}\mathbb{E}[V_1(n)]/n= \Vert \mu\Vert$. The following result extends this to stable random walks and to higher dimensions.

\begin{theorem}\label{m=1}
Assume \eqref{DOM-ATT}. Let $\alpha >1$ and $\mu \neq0$.
	It holds 
	$$
	\frac{V_1(n)}{n}\xrightarrow[n\nearrow\infty]{\mathrm{L}^1}\lVert\mu\rVert\qquad\text{and}\qquad \frac{P(n)}{n}\xrightarrow[n\nearrow\infty]{\mathrm{L}^1} \frac{\mu}{2}.
	$$
	 In particular, 
	 \begin{equation}\label{M=1}
	\lim_{n\to \infty} \frac{\mathbb{E}\bigl[V_1(n)\bigr]}{n}
	=
	\lVert\mu\rVert,\qquad
	\lim_{n\to \infty} \frac{\mathbb{E}\bigl[\lVert P(n)\rVert\bigr]}{n} =
	 \frac{\lVert\mu\rVert}{2}\qquad \text{and}\qquad \lim_{n\to \infty} \frac{\mathbb{E}\bigl[ P(n)\bigr]}{n} =
	 \frac{\mu}{2}.
	 \end{equation}
		\end{theorem}
\begin{proof}
In view of Theorem \ref{V_1_p_drift} we can base upon uniform integrability argument.
It suffices to show that $\{V_1(n)/n\}_{n\ge1}$ and $\{\lVert P(n)\rVert/n\}_{n\ge1}$ are uniformly bounded in $\mathrm{L}^{\alpha-\varepsilon}$ for some $\varepsilon>0$ with $\alpha-\varepsilon>1$, see \cite[Proposition 3.12]{Kallenberg}. 
By an analogous reasoning as in Proposition  \ref{PER} it is then enough to show that $\{ S^{(k)}(n)/n\}_{n\ge1}$, $k=1,\dots,d$, are uniformly bounded in $\mathrm{L}^{\alpha-\varepsilon}$. 
	We have $$\sup_{n\ge 1}\mathbb{E}\bigl[|S^{(k)}(n)/n|^{\alpha-\varepsilon}\bigr]\le \sup_{n\ge 1}\frac{1}{n}\sum_{i=1}^n\mathbb{E}\bigl[|Y_i^{(k)}|^{\alpha-\varepsilon}\bigr]=\mathbb{E}\bigl[|Y_1^{(k)}|^{\alpha-\varepsilon}\bigr],$$ where in the first step we used the following elementary inequality: for $a_1,\dots,a_n\ge0$ and $p>1$ it holds 
 \begin{equation}\label{eq:p}
 \Big(\sum_{i=1}^na_i\Big)^p\le\begin{cases}
	\sum_{i=1}^na_i^p, & p\in(0,1),\\
	n^{p-1}\sum_{i=1}^na_i^p, & p\ge1.
	\end{cases}
	\end{equation}
	This finishes the proof.
		\end{proof}

We next show that if higher moments of $\{\bS(n)\}_{n\ge0}$ (of order at least $2$, which implies that $\alpha =2$ and the limit process $\{\bX(t)\}_{t\ge0}$ is a Brownian motion) are finite then mean intrinsic volumes of the rescaled (through a linear mapping) convex hull converge to the convex hull of the rescaled time-space Brownian motion $\{\widetilde{\bX}(t)\}_{t\ge0}$ defined in \eqref{process tilde X}.
This leads to an asymptotic result for the mean volume $V_d$ of $\conv \SRange$. The result should be compared with \cite[Proposition 3.4]{Wade-Xu_SPA}

\begin{theorem}\label{m>1}
Assume that $\mathbb{E}[\lVert \bY_1\rVert^{2\vee mp}]<\infty$ for some $p>1$ and $\mu \neq 0$. 
Then, for any $m\in\{1,2,\dots,d\}$,
\begin{align}\label{Expext_V_m_finite}
\mathbb{E}[V_m(\mathrm{conv}\, \widetilde\bX[0,1])]<\infty
\end{align}
and, for $D_n = \mathrm{diag}(1/n,1/\sqrt{n},\dots,1/\sqrt{n})$,
\begin{align}\label{mean_V_m_drift_conv}
\lim_{n\to \infty}
\mathbb{E}\left[V_m\left(D_n\bigl(\mathrm{conv}\{\bS(0),\ldots,\bS(n)\}\bigr)\right)\right]
=
\mathbb{E}\bigl[V_m\bigl(\mathrm{conv}\, \widetilde\bX[0,1]\bigr)\bigr].
\end{align}
In particular,
\begin{equation}\label{m=d}
\lim_{n\to\infty}
\frac{\mathbb{E}\bigl[V_d(n)\bigr]}{n^{(d+1)/2}}
=
\mathbb{E}\bigl[V_d\bigl(\mathrm{conv}\, \widetilde\bX[0,1]\bigr)\bigr].
\end{equation}
	\end{theorem}
\begin{proof} 
In view of the moment assumption, $\{\bX(t)\}_{t\ge0}$ is necessarily a Brownian motion and  $b_n=\sqrt{n}$, see \cite[Theorem 2.6.6]{Ibragimov-Linnik}. 
We start by finding a polytope which bounds the convex hull $\conv \SRange$. We have
	 $$
	 \mathrm{conv}\{\bS(0),\ldots,\bS(n)\}\subseteq 
	 \bigl[\lambda_1(n),\Lambda_1(n)\bigr]
	 \times\cdots\times
	 \bigl[\lambda_d(n),\Lambda_d(n)\bigr],
	 $$
	 where $\lambda_i(n)=\min_{0\le k\le n}\langle\bS(k),\bar e_i\rangle$
	 and $\Lambda_i (n ) = \max_{0\le k\le n}\langle\bS(k),\bar e_i\rangle$.
	Hence,
	\begin{align*}
	D_n\left( \mathrm{conv}\{\bS(0),\ldots,\bS(n)\}\right)
	\subseteq 
	D_n\left(\bigl[\lambda_1(n),\Lambda_1(n)
	\bigr]\times\cdots\times
	\bigl[\lambda_d(n),\Lambda_d(n)\bigr]\right)
	\end{align*}
Further, since $D_n$ is a diagonal matrix and due to monotonicity of intrinsic volumes, 
\begin{align*}
V_m\left(D_n\bigl( \mathrm{conv}\{\bS(0),\ldots,\bS(n)\}\bigr)\right)
&\le V_m\left( \bigl[\lambda_1(n)/n,\Lambda_1(n)/n\bigr] \times\cdots \times
\bigl[\lambda_d(n)/\sqrt{n},\Lambda_d(n)/\sqrt{n}\bigr]
 \right).
\end{align*}
According to \cite[Proposition 5.5]{Lotz-McCoy-Nourdin-Peccati-Tropp-2020},
the $m$-th intrinsic volume in the right hand side of the last inequality 
is the coefficient next to $w^m$ of the polynomial $w\mapsto \prod_{i=1}^d(1+ \upsilon_i(n) w )$,
where
\begin{align*}
\upsilon_{1}(n)=\Lambda_1(n)/n-\lambda_1(n)/n\qquad \mathrm{and}\qquad
\upsilon_i(n)=\Lambda_i(n)/\sqrt{n}-\lambda_i(n)/\sqrt{n},\quad i=2,\dots,d.
\end{align*}
Hence, 
\begin{align*}
&V_m\left(D_n\bigl(\mathrm{conv}\{\bS(0),\ldots,\bS(n)\}\bigr)\right)\\
&\le \sum_{i_1,\dots,i_m\in\{1,\dots,d\}}\prod_{j=1}^m\upsilon_{i_j}(n)
=\upsilon_{1}(n)\!\!\!\!\!
\sum_{i_1,\dots,i_{m-1}\in\{2,\dots,d\}}\prod_{j=1}^{m-1}\upsilon_{i_j} (n)\, +\!\! \sum_{i_1,\dots,i_m\in\{2,\dots,d\}}\prod_{j=1}^m \upsilon_{i_j}(n),
\end{align*}
where we use convention that the empty product is equal to 1.
Clearly 
\begin{align*}
\upsilon_1(n)\leq 2\max_{0\le k\le n}|\langle\bS(k),\bar e_1\rangle|/n\qquad \mathrm{and}\qquad \upsilon_i(n)\leq 2\max_{0\le k\le n}|\langle\bS(k),\bar e_i\rangle|/\sqrt{n}.
\end{align*}
This yields
\begin{align*}
V_m\left(D_n\bigl(\mathrm{conv}\{\bS(0),\ldots,\bS(n)\}\bigr)\right)
&\le 
\frac{2^m}{n^{(m+1)/2}}\max_{0\le k\le n}|\langle\bS(k),\bar e_1\rangle| \sum_{i_1,\dots,i_{m-1}\in\{1,\dots,d\}}\prod_{j=1}^{m-1}\max_{0\le k\le n}|\langle\bS(k),\bar e_{i_j}\rangle|\\
&\qquad+\frac{2^m}{n^{m/2}}\sum_{i_1,\dots,i_m\in\{2,\dots,d\}}\prod_{j=1}^m \max_{0\le k\le n}|\langle\bS(k),\bar e_{i_j}\rangle|.
\end{align*}
H\"{o}lder's inequality implies that for any $q>1$,
\begin{align*}
\mathbb{E}\Big[\prod_{j=1}^m \max_{0\le k\le n}|\langle\bS(k),\bar e_{i_j}\rangle|^q\Big]\le \prod_{j=1}^m \big(\mathbb{E}\bigl[\bigl(\max_{0\le k\le n}|\langle\bS(k),\bar e_{i_j}\rangle|\bigr)^{mq}\bigr]\big)^{1/m}.
\end{align*}
Without loss of generality we can assume $p\in(1,2)$.
According to \cite[Lemma A.1]{Wade-Xu_SPA}
it holds\footnote{We use the standard  $\mathcal{O}$-notation: for $f:\N\to\R$ and $g:\N\to(0,\infty)$ we write $f(n)=\mathcal{O}(g(n))$ if, and only if,  there is a constant $c>0$ such that $|f(n)|\le c g(n)$ for all $n\in\N$.} 
$$\left( \mathbb{E}\bigl[\bigl(\max_{0\le k\le n}|\langle\bS(k),\bar e_i\rangle|\bigr)^{mq}\bigr]\right)^{1/m}=\begin{cases}
\mathcal{O}(n^{2}), & m=1, q=2\ \text{and}\ i=1,\\
\mathcal{O}(n^{p}), & m\ge2,\ q=p\ \text{and}\ i=1,\\
\mathcal{O}(n), & m=1,\ q=2 \ \text{and}\ i\in\{2,\dots,d\},\\
\mathcal{O}(n^{p/2}), & m\ge2,\ q=p\ \text{and}\ i\in\{2,\dots,d\}.
\end{cases}
$$ 
From this we infer that 
$$\mathbb{E}\left[V_m\left(D_n\bigl(\mathrm{conv}\bigl(\{\bS(0),\ldots,\bS(n)\}\bigr)\bigr)\right)^q\right]=
\mathcal{O}(1).$$
It follows that the sequence $\{V_m(D_n(\mathrm{conv}\{\bS(0),\ldots,\bS(n)\}))\}_{n\ge0}$ is uniformly integrable. Thus, in view of Theorem \ref{CONV-DRIFT-2} and \cite[Lemma 3.11]{Kallenberg} we obtain \eqref{Expext_V_m_finite} and \eqref{mean_V_m_drift_conv}. Equation \eqref{m=d} follows from the fact that 
$$
V_d \big(D_n (\mathrm{conv}\{\bS(0),\ldots,\bS(n)\})\big)=
\det (D_n)\, V_d(\conv \SRange)=
\frac{V_d(n)}{n^{(d+1)/2}},
$$
and the proof is finished.
	\end{proof}

\begin{remark}
(i) We note that the fact that $\mathbb{E}[V_m(\mathrm{conv}\, \widetilde\bX[0,1])]<\infty$ follows also by \cite[Theorem 1.1]{Molchanov-Wespi}.\\
(ii)  The arguments used in Theorem \ref{m>1} do not apply in the case when $m=1$ and $\alpha<2$ as in such case we would require that $\mathbb{E}[\lVert \bY_1\rVert^{\alpha}]<\infty$. Then, by the same reasoning as in  Theorem \ref{m>1} we would obtain 
		$$\mathbb{E}\bigl[\bigl(\max_{0\le k\le n}|\langle\bS(k),\bar e_i\rangle|\bigr)^{\alpha}\bigr]=\begin{cases}
		\mathcal{O}(n^{\alpha}), &  i=1,\\
		\mathcal{O}(n), &  i\in\{2,\dots,d\}.
		\end{cases}
		$$ 
		This would imply 
		$$
		\mathbb{E}\left[V_1\left(D_n\bigl(\mathrm{conv}\{\bS(0),\ldots,\bS(n)\}\bigr)\right)^\alpha\right]=
		\mathcal{O}\bigr(n/b_n^\alpha\bigl).
		$$ 
		In order to infer uniform integrability we would have to assume that $\limsup_{n\to\infty}n/b_n^\alpha<\infty.$ 
		However, this is in contradiction with \cite[Theorem]{Tucker-1975}.
	\end{remark}

Before we continue our discussion on convergence of mean intrinsic volumes of the convex hull of the random walk, we find the precise value of the limit in \eqref{m=d} which results in deriving a formula for the expected volume of the convex hull spanned by a time-space Brownian motion run up to time one. For this reason, we extend the argument from \cite[Proposition 3.4]{Wade-Xu_SPA} where the planar case was handled. The main idea is to construct a specific Gaussian random walk such that its convex hull approximates the convex hull of the time-space Brownian motion. For such random walk we can then compute the expected volume of the convex hull through a combinatorial formula given in \cite{Vysotsky}.

\begin{theorem}\label{ST}
Let $\widetilde {\mathbf{B}}(t)=(t,\mathbf{B}(t))$, $t\ge0$, where $\{\mathbf{B}(t)\}_{t\ge0}$ is a standard Brownian motion in $\mathbb{R}^{d-1}$.
	It holds 
	$$
	\bbE\left[\mathrm{Vol}_d\bigl(\mathrm{conv}\, \widetilde {\mathbf{B}}[0,1]\bigr)\right]=\frac{2^{(d+1)/2}\pi^{(d-1)/2}}{(d+1)!}.
	$$
\end{theorem}
\begin{proof}
	 Let $\{\mathbf{Z}_i\}_{i\in\N}$ be a sequence of independent and identically distributed random vectors with law $\mathrm{N}(0,\mathrm{I}_{d-1})$. 
	Here, $\mathrm{I}_{d-1}$ stands for the $(d-1)\times(d-1)$ identity matrix.	 
	 Let $\widetilde {\mathbf{Z}}_i=(1,\mathbf{Z}_i)$ and denote by  $\widetilde {\bS}(n)=\sum_{i=1}^n\widetilde{\mathbf{Z}}_i$ the corresponding random walk such that $ \widetilde {\bS}(0)=0.$ 
	Evidently, $\mathbb{E}[\widetilde{\mathbf{Z}}_1]=e_1$, where $\{e_1,\ldots ,e_d\}
	$ is the standard orthonormal basis of $\mathbb{R}^d$.
	According to \cite[Theorem 6.13]{McRedmond_survey} the following convergence holds 
	$$
	\frac{\mathrm{Vol}_d\bigl(\mathrm{conv}\,\{\widetilde{\bS}(0),\dots,\widetilde{\bS}(n)\}\bigr)}{n^{(d+1)/2}}\xrightarrow[n\nearrow\infty]{\mathcal{D}} \mathrm{Vol}_d\bigl(\mathrm{conv}\,\widetilde {\mathbf{B}}[0,1]\bigr).
	$$ 
 We can apply the same reasoning as in the proof of Theorem \ref{m>1} to show that
	$$\mathbb{E}\left[\mathrm{Vol}_d\bigl(\mathrm{conv}\bigl(\{\widetilde{\bS}(0),\ldots,\widetilde{\bS}(n)\}\bigr)\bigr)^2\right]=
	\mathcal{O}(n^{d+1}).
	$$ 
This implies uniform integrability and 
	by \cite[Lemma 3.11]{Kallenberg} we then infer that
	$$
	\bbE\left[\mathrm{Vol}_d\bigl(\mathrm{conv}\,\widetilde{\mathbf{B}}[0,1]\bigr)\right]=\lim_{n\to\infty}n^{-(d+1)/2}
	\bbE\left[\mathrm{Vol}_d\bigl(\mathrm{conv}\,\{\widetilde{\bS}(0),\dots,\widetilde{\bS}(n)\}\bigr)\right].
	$$
	We are left to compute the limit in the last expression.	
	Let $\{\widetilde{\bS}^{(k)}(n)\}_{n\ge0}$, for $k=1,\dots,d$, be independent copies of 
	$\{\widetilde{\bS}(n)\}_{n\ge0}$.
	According to \cite[Corrolary 2]{Vysotsky} it holds	
	$$
	\bbE\left[\mathrm{Vol}_d\bigl(\mathrm{conv}\,\{\widetilde{\bS}(0),\dots,\widetilde{\bS}(n)\}\bigr)\right]
	=
	\frac{1}{d!}\sum_{\substack{j_1+\cdots+j_d\le n \\ j_1,\dots,j_d\in\N}}\frac{\bbE\left[\left|\det \big(\widetilde{\bS}^{(1)}(j_1)\cdots\widetilde{\bS}^{(d)}(j_d)\big)\right|\right]}{j_1\cdots j_d}.
	$$	
	For fixed $j_1,\dots,j_d\in\N$ we have
	$$
	\det\bigl(\widetilde{\bS}^{(1)}(j_1)\cdots\widetilde{\bS}^{(d)}(j_d)\bigr)=
	\det 
	\begin{pmatrix}
	j_1 & \cdots & j_d \\ 
	\sum_{i=1}^{j_1} \mathbf{Z}^{(1)}_i & \cdots & \sum_{i=1}^{j_d} \mathbf{Z}^{(d)}_i 
	\end{pmatrix},
	$$
	where $\{\mathbf{Z}_i^{(k)}\}_{i\in\N}$, for $k=1,\ldots ,d$, are independent copies of 
	$\{\mathbf{Z}_i\}_{i\in\N}$.
	Clearly, $\sum_{i=1}^{j_k} \mathbf{Z}^{(k)}_i$ is of the form $\sqrt{j_k}\mathbf{W}_k$, where $\mathbf{W}_k=(\mathbf{W}_k^{(1)},\dots,\mathbf{W}_k^{(d-1)})$ are independent random vectors with law $\mathrm{N}(0,\mathrm{I}_{d-1})$.
	Hence, 
	\begin{align*}
	\det\bigl(\widetilde{\bS}^{(1)}(j_1)\cdots\widetilde{\bS}^{(d)}(j_d)\bigr)&
	=\sqrt{j_1\cdots j_d}
	\det 
	\begin{pmatrix}
	\sqrt{j_1} & \cdots & \sqrt{j_d} \\ 
	\mathbf{W}_1 & \cdots & \mathbf{W}_d
	\end{pmatrix}\\&=\sqrt{j_1\cdots j_d}\sqrt{j_1+\cdots+j_d}
	\det 
	\begin{pmatrix}
	\frac{\sqrt{j_1}}{\sqrt{j_1+\cdots+j_d}} &  \mathbf{W}_1 \\ 
	\vdots & \vdots 	\\
	\frac{\sqrt{j_d}}{\sqrt{j_1+\cdots+j_d}}  & \mathbf{W}_d
	\end{pmatrix}.
	\end{align*}
	We note that the first column in the last determinant is a unit vector which we denote by $v$. Let $O$ be an  orthogonal matrix such that $Ov=e_1$. Further, the vectors (forming the remaining columns of the last determinant) $(\mathbf{W}_1^{(k)},\dots,\mathbf{W}_d^{(k)})$, $k=1,\dots,d-1$, are independent with law  $\mathrm{N}(0,\mathrm{I}_{d})$. It follows that the 
	vectors $\widetilde{\mathbf{W}}_k=O(\mathbf{W}_1^{(k)},\dots,\mathbf{W}_d^{(k)})^{T}$, $k=1,\dots,d-1$, are also independent and have  law $\mathrm{N}(0,\mathrm{I}_{d})$. Consequently, 
	\begin{align*}
		\left|\det\bigl(\widetilde{\bS}^{(1)}(j_1)\cdots\widetilde{\bS}^{(d)}(j_d)\bigr)\right|
		&=
		\sqrt{j_1\cdots j_d}\sqrt{j_1+\cdots+j_d}\left|\det\bigl(
	e_1\ \widetilde{\mathbf{W}}_1\cdots\widetilde{\mathbf{W}}_{d-1}
	\bigr)\right|\\
	&\stackrel{\mathcal{D}}{=}
	\sqrt{j_1\cdots j_d}\sqrt{j_1+\cdots+j_d}\left|\det\bigl(
	\mathbf{W}_1\cdots \mathbf{W}_{d-1}\bigr)\right|.
	\end{align*}
	The distribution of the last determinant was found in \cite{Goodman} and it is given by
	 $$
	 \left|\det\bigl(
	\mathbf{W}_1\cdots \mathbf{W}_{d-1}\bigr)\right| \stackrel{\mathcal{D}}{=}\sqrt{\chi^2(1)\cdots\chi^2(d-1)},   	$$ 
	where $\chi^2(k)$, for $k=1,\ldots ,d-1$, are independent  chi-squared random variables with $k$ 
	degrees of freedom, respectively. 
	Since, 
	$$
	\bbE\left[\sqrt{\chi^2(k)}\right]=\sqrt{2}\frac{\Gamma((k+1)/2)}{\Gamma(k/2)},
	$$ 
	we obtain 
	$$
	\bbE\left[\left|\det\bigl(\widetilde{\bS}^{(1)}(j_1)\cdots\widetilde{\bS}^{(d)}(j_d)\bigr)\right|\right]=\sqrt{j_1\cdots j_d}\sqrt{j_1+\cdots+j_d}\, \frac{2^{(d-1)/2}}{\sqrt{\pi}}\Gamma(d/2)
	$$
	and whence 
	$$
	\bbE\left[\mathrm{Vol}_d\bigl(\mathrm{conv}\,\{\widetilde{\bS}(0),\dots,\widetilde{\bS}(n)\}\bigr)\right]=\frac{2^{(d-1)/2}\Gamma(d/2)}{d!\sqrt{\pi}}\sum_{\substack{j_1+\cdots+j_d\le n \\ j_1,\dots,j_d\in\N}}\frac{\sqrt{j_1+\cdots+j_d}}{\sqrt{j_1\cdots j_d}}.
	$$	
	We finally claim that 
	$$
	\lim_{n\to\infty}\frac{1}{n^{(d+1)/2}}\sum_{\substack{j_1+\cdots+j_d\le n \\ j_1,\dots,j_d\in\N}}\frac{\sqrt{j_1+\cdots+j_d}}{\sqrt{j_1\cdots j_d}}=\frac{2\pi^{d/2}}{(d+1)\Gamma(d/2)}.
	$$ 
	We consider two sequences 
	$$
	a_n=\sum_{\substack{j_1+\cdots+j_d\le n \\ j_1,\dots,j_d\in\N}}\frac{\sqrt{j_1+\cdots+j_d}}{\sqrt{j_1\cdots j_d}}\qquad\text{and}\qquad f_n=n^{(d+1)/2}.
	$$
	 Since $\{f_n\}_{n\in\N}$ is strictly monotone and divergent, according to Stolz-Ces\`{a}ro theorem \cite[Theorem 1.22]{Marian} it suffices to show that $$\lim_{n\to\infty}\frac{a_{n+1}-a_n}{f_{n+1}-f_n}=\frac{2\pi^{d/2}}{(d+1)\Gamma(d/2)}.$$ 
	 Since $$\frac{(n+1)^u-n^u}{un^{u-1}}=1$$ for all $u>0$, we have 
	 $$\lim_{n\to\infty}\frac{f_{n+1}-f_n}{n^{(d-1)/2}}=\frac{d+1}{2}.$$ 
	 Hence, the claim will follow if we find the limit
	 \begin{align*}
	 \lim_{n\to\infty}\frac{a_{n+1}-a_n}{n^{(d-1)/2}}&=\lim_{n\to\infty}\frac{1}{n^{(d-1)/2}}\sum_{\substack{j_1+\cdots+j_d= n+1 \\ j_1,\dots,j_d\in\N}}\frac{\sqrt{j_1+\cdots+j_d}}{\sqrt{j_1\cdots j_d}}\\&=\lim_{n\to\infty}\frac{1}{n^{(d-2)/2}}\sum_{\substack{j_1+\cdots+j_d= n+1 \\ j_1,\dots,j_d\in\N}}(j_1\cdots j_d)^{-1/2}.
	 \end{align*}
	 The last expression is equal to the integral sum converging to the Dirichlet integral given in terms of the multinomial Beta function, see \cite[page 11]{Shui}. We have 
	 \begin{align*}
	\lim_{n\to\infty}\frac{1}{n^{(d-2)/2}}\sum_{\substack{j_1+\cdots+j_d= n+1 \\ j_1,\dots,j_d\in\N}}(j_1\cdots j_d)^{-1/2}&=\lim_{n\to\infty}\sum_{\substack{j_1+\cdots+j_d= n+1 \\ j_1,\dots,j_d\in\N}}\left(\frac{j_1\cdots j_d}{n^d}\right)^{-1/2} \frac{1}{n^{d-1}}\\
	&=\int_{\substack{u_1+\cdots+u_d= 1 \\ u_1,\dots,u_d\ge0}}(u_1\cdots u_d)^{-1/2}\mathrm{d}u_1\dots \mathrm{d}u_d\\
	&=\frac{\prod_{i=1}^d\Gamma(1/2)}{\Gamma(\sum_{i=1}^d1/2)}=\frac{\pi^{d/2}}{\Gamma(d/2)}
	\end{align*} 
	and the proof is finished.
\end{proof}

We next show how Theorem \ref{ST} provides the precise form of the limit in \eqref{m=d}. Recall Remark \ref{REM}. We have $$\bigl\{\widetilde{\bX}(t)\bigr\}_{t\ge0}\stackrel{\mathcal{D}}{=}\big\{\Sigma_{\mu}^{1/2}\widetilde{\mathbf{B}}(t)\big\}_{t\ge0} 
\qquad \text{and}\qquad 
V_d\bigl(\mathrm{conv}\, \widetilde\bX[0,1]\bigr)
=
\lVert\mu\rVert\sqrt{\det (\Sigma_{\bot})}\, V_d\bigl(\mathrm{conv}\, \widetilde{\mathbf{B}}[0,1]\bigr),
$$
where $\Sigma_{\bot}=\bbE[(\bar{\bY}^\bot_1)(\bar{\bY}^\bot_1)^t]$ and $$\Sigma_{\mu}=\begin{pmatrix}
\lVert\mu\rVert^2 & 0 \\
0 & \Sigma_{\bot}
\end{pmatrix}.$$

In view od Theorems \ref{m>1} and \ref{ST} we obtain the following result.
\begin{corollary}
	Under the assumptions of Theorem \ref{m>1},
$$
\lim_{n\to \infty}
\mathbb{E}\left[V_m\left(D_n\bigl(\mathrm{conv}\{\bS(0),\ldots,\bS(n)\}\bigr)\right)\right]
=
\mathbb{E}\bigl[V_m\bigl(\Sigma_{\mu}^{1/2}\mathrm{conv}\, \widetilde{\mathbf{B}}[0,1]\bigr)\bigr].
$$	
		and  $$
	\lim_{n\to\infty}
	\frac{\mathbb{E}\bigl[V_d(n)\bigr]}{n^{(d+1)/2}}
	=
	\frac{2^{(d+1)/2}\pi^{(d-1)/2}}{(d+1)!}\lVert\mu\rVert \sqrt{\det (\Sigma_{\bot})}.
	$$
\end{corollary}

We finally aim to derive an analogue of formula \eqref{M=1} or \eqref{m=d} for mean intrinsic volumes with $m\in\{2,\dots,d-1\}$.
Before we formulate and prove the result, we need some preparation. 
We assume that $\mu \neq 0$ and
$\Prob(\bY_1\in \mathbb{h})=0$ for any affine hyperplane $\mathbb{h}\subset\R^d$. Moreover, let $\mathbb{E}[\lVert \bY_1\rVert^{mp}]<\infty$ for some $p>1$ (recall that this implies that $\{\bX(t)\}_{t\ge0}$ is a Brownian motion and $b_n=\sqrt{n}$). 
We fix $m\in\{2,\dots,d-1\}$ and let 
$P_m$ be the orthogonal projector from $\R^d$ onto $\R^m$. For $Q\in\mathrm{SO}(d)$ (the space of all orthogonal matrices with determinant one) it holds $$
P_mQ\, \mathrm{conv}\{\bS(0),\ldots,\bS(n)\}=\mathrm{conv}\{P_mQ\bS(0),\ldots,P_mQ\bS(n)\}.
$$ 
Clearly, $P_mQ\bS(n)=\sum_{i=1}^nP_mQ\bY_i$ is a $\R^m$-valued random walk.
Also, $\mathbb{E}[P_mQ\bY_1]=P_mQ\mu$ and 
$$
\frac{P_mQ\bS(n)-nP_mQ\mu}{\sqrt{n}}\xrightarrow[n\nearrow\infty]{\mathcal{D}} P_mQ\bX(1).
$$ 
The process $\{P_mQ\bX(t)\}_{t\ge0}$ is a Brownian motion in $\R^m$.
Thus, analogously as in  \eqref{FCLT}, 
$$
\left\{\frac{P_mQ\bS(\lfloor nt\rfloor)-\lfloor nt\rfloor P_mQ\mu} {\sqrt{n}}\right\}_{t\ge0}\xrightarrow[n\nearrow\infty]{\mathcal{J}^{d}_1}\{P_mQ\bX(t)\}_{t\ge0}.
$$
If $P_mQ\mu=0$ then, by Theorem \ref{Vm},
\begin{equation}\label{mmu><0}
\lim_{n\to \infty} 
\frac{\bbE\bigl[\mathrm{Vol}_m\bigl(P_mQ\, \mathrm{conv}\{\bS(0),\ldots,\bS(n)\}\bigr)\bigr]}{n^{(m+1)/2}}
=0.	
\end{equation}
We next consider the case 
when $P_mQ\mu\neq0$.
Let	$\{ e^{Q}_1,\dots, e^{Q}_m\}$ be an  orthonormal basis of $\R^m$ such that $ e^{Q}_1=P_mQ\mu/\lVert P_m Q\mu\rVert$. Similarly as before, 
let $\phi_Q^\bot:\R^m\to\R^{m-1}$ be a linear mapping given by  $\phi^\bot_Q(x)=(\langle x, e^Q_2\rangle,\dots, \langle x, e^Q_m\rangle)$.
 Clearly, $\{\phi_Q^\bot(P_mQ\bS(n))\}_{n\ge0}$ is a  zero-drift random walk in $\R^{m-1}$ satisfying  
 $$
 \frac{\phi_Q^\bot(P_mQ\bS(n))}{\sqrt{n}}\xrightarrow[n\nearrow\infty]{\mathcal{D}} \phi_Q^\bot\bigl(P_mQ\bX(1)\bigr).$$ 
 The process
$\{\phi_Q^\bot(P_mQ\bX(t))\}_{t\ge0}$ is a Brownian motion in $\R^{m-1}$. 
We define the following time-space Brownian motion in $\R^m$
\begin{equation}\label{X_Q_space-ime}
\widetilde{\bX}_{Q,m}(t)=\bigl(\lVert P_mQ\mu\rVert t,\phi_Q^\bot(P_mQ\bX(t))\bigr),\qquad t\geq 0.
\end{equation}
By Theorem \ref{m>1}, 
\begin{equation}\label{mmu><00}
\lim_{n\to \infty}
\frac{\bbE\bigl[\mathrm{Vol}_m\bigl(P_mQ\, \mathrm{conv}\{\bS(0),\ldots,\bS(n)\}\bigr)\bigr]}{n^{(m+1)/2}}
=
\mathbb{E}\bigl[\mathrm{Vol}_m\bigl(\mathrm{conv}\, \widetilde\bX_{Q,m}[0,1]\bigr)\bigr].
\end{equation}
We finally claim that
\begin{equation}\label{X}
\sup_{Q\in\mathrm{SO}(d)}\mathbb{E}\bigl[\mathrm{Vol}_m\bigl(\mathrm{conv}\, \widetilde\bX_{Q,m}[0,1]\bigr)\bigr]
=
\sup_{ Q\in\mathrm{SO}(d):\, P_mQ\mu\neq0}\mathbb{E}\bigl[\mathrm{Vol}_m\bigl(\mathrm{conv}\,\widetilde\bX_{Q,m}[0,1]\bigr)\bigr]<\infty.
\end{equation}
Indeed,
we have
\begin{align*}
&\mathrm{conv}\,\widetilde\bX_{Q,m}[0,1]\\&\subseteq \bigl[0,\lVert P_m Q  \mu\rVert\bigr]\times\bigl[\inf_{0\le t\le 1}\langle P_mQ\bX(t), e^Q_2\rangle,\sup_{0\le t\le 1}\langle P_mQ\bX(t),\ e^Q_2\rangle\bigr]\\&\qquad \times\cdots\times\bigl[\inf_{0\le t\le 1}\langle P_mQ\bX(t), e^Q_m\rangle,\sup_{0\le t\le 1}\langle P_mQ\bX(t),\ e^Q_m\rangle\bigr]\\
&\subseteq
\bigl[0,\lVert   \mu\rVert\bigr]\times\bigl[-\sup_{0\le t\le 1}\lVert\bX(t)\rVert,\sup_{0\le t\le 1}\lVert\bX(t)\rVert\bigr] \times\cdots\times\bigl[-\sup_{0\le t\le 1}\lVert\bX(t)\rVert,\sup_{0\le t\le 1}\lVert\bX(t)\rVert\bigr]\bigr].\end{align*}
This yields
\begin{align*}
\bbE\bigl[\mathrm{Vol}_m\big(\mathrm{conv}\, \widetilde\bX_{Q,m}[0,1]\big)\bigr]\le 2^{m-1}\lVert\mu\rVert\bbE\bigl[\sup_{0\le t\le 1}\lVert\bX(t)\rVert^{m-1}\bigr]
\end{align*}
and by \eqref{eq:p} we obtain
\begin{align*}
\bbE\bigl[\mathrm{Vol}_m\big(\mathrm{conv}\, \widetilde\bX_{Q,m}[0,1]\big)\bigr]
\le 2^{m-1}\lVert\mu\rVert\begin{cases}
\sum_{k=1}^d\bbE\bigl[\sup_{0\le t\le 1}\rvert X^{(k)}(t)\rvert^{m-1}\bigr], & m=2,\\
d^{(m-3)/2}\sum_{k=1}^d\bbE\bigl[\sup_{0\le t\le 1}\rvert X^{(k)}(t)\rvert^{m-1}\bigr], & m\ge3.
\end{cases}
\end{align*}
Since $\{\bX(t)\}_{t\ge0}$ is a zero-drift Brownian motion,
the claim follows by Doob's maximal inequality. 

We are ready to state and prove the theorem.

\begin{theorem}\label{Em}
Assume that $\mu\neq 0$ and $\mathbb{E}[\lVert \bY_1\rVert^{ mp}]<\infty$ for some $p>1$. Further, suppose that $\Prob(\bY_1\in \mathbb{h})=0$ for any affine hyperplane $\mathbb{h}\subset\R^d$. Then, for $m\in\{2,\dots,d-1\}$, 
$$
\lim_{n\to \infty} \frac{\mathbb{E}\bigl[V_m(n)\bigr]}{n^{(m+1)/2}}
=
\binom{d}{m}\frac{\kappa_d}{\kappa_m\kappa_{d-m}}\int_{\mathrm{SO}(d)}\mathbb{E}\bigl[\mathrm{Vol}_m\bigl(\mathrm{conv}\, \widetilde\bX_{Q,m}[0,1]\bigr)\bigr]\nu(\mathrm{d}Q),
$$ 
where $\{\widetilde\bX_{Q,m}(t)\}_{t\ge0}$ is  defied in \eqref{X_Q_space-ime} and  
$\nu(\mathrm{d}Q)$ denotes  the  (probability) Haar measure on the special orthogonal group $\mathrm{SO}(d)$.
\end{theorem}
\begin{proof} 
We make use of Kubota's formula which asserts that
\begin{equation}\label{Kubota}
V_m(n)=\binom{d}{m}\frac{\kappa_d}{\kappa_m\kappa_{d-m}}\int_{\mathrm{SO}(d)}\mathrm{Vol}_m\bigl(P_mQ\, \mathrm{conv}\{\bS(0),\ldots,\bS(n)\}\bigr)\nu(\mathrm{d}Q),
\end{equation} 
see \cite[Theorem 1.3]{Schneider} and \cite{Lotz-McCoy-Nourdin-Peccati-Tropp-2020}.
 From \eqref{mmu><0}, \eqref{mmu><00} and \eqref{X} it follows that 
 \begin{align*}
 &\int_{\mathrm{SO}(d)}\lim_{n\to\infty}\frac{\bbE\bigl[\mathrm{Vol}_m\bigl(P_mQ\, \mathrm{conv}\{\bS(0),\ldots,\bS(n)\}\bigr)\bigr]}{n^{(m+1)/2}}\nu(\mathrm{d}Q)\\
&=\int_{\{Q\in\mathrm{SO}(d):P_mQ\mu\neq0\}}\mathbb{E}\bigl[V_m\bigl(\mathrm{conv}\, \widetilde\bX_{Q,m}[0,1]\bigr)\bigr]\nu(\mathrm{d}Q)\\
&=\int_{\mathrm{SO}(d)}\mathbb{E}\bigl[V_m\bigl(\mathrm{conv}\, \widetilde\bX_{Q,m}[0,1]\bigr)\bigr]\nu(\mathrm{d}Q).
\end{align*}
Hence, we are left to justify interchange of the limit and the integral in \eqref{Kubota}.
This will be possible if we show that 
$$
\sup_{n\in\N}\sup_{Q\in\mathrm{SO}(d)}\frac{\bbE\bigl[\mathrm{Vol}_m\bigl(P_mQ\, \mathrm{conv}\{\bS(0),\ldots,\bS(n)\}\bigr)\bigr]}{n^{(m+1)/2}}<\infty.
$$
We can proceed exactly in the same way as in the proof of Theorem \ref{m>1} but this time we work with $\mathrm{Vol}_m$ instead of $V_m$ and with the walk $\{P_mQ\bS(n)\}_{n\ge0}$.
If $P_mQ\mu\neq0$ we use the basis $\{e_1^Q,\ldots ,e_m^Q\}$ of $\R^m$ and obtain
\begin{align*}
\mathrm{Vol}_m\left(P_mQ\, \mathrm{conv}\{\bS(0),\ldots,\bS(n)\}\right)\le &2^m\prod_{i=1}^{m}\max_{0\le k\le n}|\langle P_mQ\bS(k), e^Q_{i}\rangle|.
\end{align*}
By H\"{o}lder's inequality,
$$
\mathbb{E}\Big[\prod_{i=1}^{m}\max_{0\le k\le n}|\langle P_mQ\bS(k), e^Q_{i}\rangle|^p\Big]\le \prod_{i=1}^m\mathbb{E}\bigl[\bigl(\max_{0\le k\le n}|\langle P_mQ\bS(k), e^Q_{i}\rangle|\bigr)^{mp}\bigr]^{1/m}.
$$
Note that $\mathbb{E}[\langle P_mQ\bS(k), e^Q_{i}\rangle]=0$ for $i\in\{2,\dots,m\}$. Thus in this case, Doob's maximal inequality entails 
$$\mathbb{E}\bigl[\bigl(\max_{0\le k\le n}|\langle P_mQ\bS(k), e^Q_{i}\rangle|\bigr)^{mp}\bigr]^{1/m}\le \frac{(mp)^p}{(mp-1)^p}\mathbb{E}\bigl[\bigl(|\langle P_mQ\bS(n), e^Q_{i}\rangle|\bigr)^{mp}\bigr]^{1/m}.
$$ 
This, together with \cite[Corollary 3.8.2]{Gut}, implies that for some $c_1=c_1(p)>0$,
\begin{align*}
\mathbb{E}\bigl[\bigl(\max_{0\le k\le n}|\langle P_mQ\bS(k), e^Q_{i}\rangle|\bigr)^{mp}\bigr]^{1/m}\le c_1\bbE\bigl[|\langle P_mQ\bY_1,e_i^Q\rangle|^{mp}\bigr]^{1/m} n^{p/2}
\le c_1 \bbE\bigl[\lVert\bY_1\rVert^{mp}\bigr]^{1/m}n^{p/2}.
\end{align*}
For $i=1$ we have 
\begin{align*}
\mathbb{E}\bigl[\bigl(\max_{0\le k\le n}|\langle P_mQ\bS(k), e^Q_{1}\rangle|\bigr)^{mp}\bigr]^{1/m}\le \bbE\left[\left(\sum_{k=1}^n\lVert\bY(k)\rVert\right)^{mp}\right]^{1/m}.
\end{align*}
By \cite[Lemma 3.9.1]{Gut}, there is $c_2=c_2(p)>0$ such that
$$\mathbb{E}\bigl[\bigl(\max_{0\le k\le n}|\langle P_mQ\bS(k), e^Q_{1}\rangle|\bigr)^{mp}\bigr]^{1/m}\le 
c_2 \mathbb{E}\bigl[\lVert \bY_1\rVert\bigr]^p n^{p}.$$
It follows that
\begin{equation}\label{bounded_Vol}
\sup_{n\in\N}\sup_{Q\in\mathrm{SO}(d): P_mQ\mu\neq0}\frac{\bbE\bigl[\mathrm{Vol}_m\bigl(P_mQ\, \mathrm{conv}\{\bS(0),\ldots,\bS(n)\}\bigr)\bigr]}{n^{(m+1)/2}}<\infty.
\end{equation}
If $P_mQ\mu=0$, then with the use of the standard basis $\{e_1,\dots,e_m\}$ and through the same reasoning we arrive at
\begin{align*}\mathbb{E}\bigl[\bigl(\max_{0\le k\le n}|\langle P_mQ\bS(k), e_{i}\rangle|\bigr)^{mp}\bigr]^{1/m}&\le c_3\bbE\bigl[|\langle P_mQ\bY_1,e_i\rangle|^{mp}\bigr]^{1/m}
n^{p/2}\\
&\ \le c_3 \bbE\bigl[\lVert\bY_1\rVert^{mp}\bigr]^{1/m}\
n^{p/2}\end{align*}
for some $c_3=c_3(p)>0$ (observe that in this case $\bbE[\langle P_mQ\bS(k), e_{i}\rangle]=0$ for all $i=1,\dots,m$).
Hence,  \eqref{bounded_Vol} is valid also for $Q\in SO(d)$ such that $P_mQ\mu=0$ and the proof is finished.
	\end{proof}

\begin{remark}
	With the same arguments as in Theorem  \ref{Em} we can obtain the corresponding result if  $m=1$ and $p=2$ or $m=d$. This would entail
	$$
	\lim_{n\to \infty} 
	\frac{\mathbb{E}\bigl[V_m(n)\bigr]}{n^{(m+1)/2}}
	=
	\begin{cases}
	\frac{d\kappa_d}{\kappa_1\kappa_{d-1}}\int_{\mathrm{SO}(d)}\mathbb{E}\bigl[V_1\bigl(\mathrm{conv}\,\widetilde\bX_{Q,1}[0,1]\bigr)\bigr]\nu(\mathrm{d}Q), &  m=1,\\
\int_{\mathrm{SO}(d)}\mathbb{E}\bigl[V_d\bigl(\mathrm{conv}\,\widetilde\bX_{Q,d}[0,1]\bigr)\bigr]\nu(\mathrm{d}Q), &  m=d.
	\end{cases}
	$$ 
Then, by Theorems \ref{m=1} and \ref{m>1} it follows that 
$$
\frac{d\kappa_d}{\kappa_1\kappa_{d-1}}\int_{\mathrm{SO}(d)}\mathbb{E}\bigl[V_1\bigl(\mathrm{conv}\,\widetilde\bX_{Q,1}[0,1]\bigr)\bigr]\nu(\mathrm{d}Q)=\lVert\mu\rVert$$ 
and 
$$\int_{\mathrm{SO}(d)}\mathbb{E}\bigl[V_d\bigl(\mathrm{conv}\, \widetilde\bX_{Q,d}[0,1]\bigr)\bigr]\nu(\mathrm{d}Q)=\mathbb{E}\bigl[V_d\bigl(\mathrm{conv}\, \widetilde\bX[0,1]\bigr)\bigr].
$$
		\end{remark}

\section{Variance asymptotics}\label{sec:Var}

In this section, we discuss the variance asymptotics of the intrinsic volumes $V_m(n)$, for $m=1,\dots,d$, under the assumption that moments (of the step of the random walk) of order higher than two are finite. The first result concerns the zero-drift case and it should be compared with \cite[Propositions 3.5]{Wade-Xu_SPA} where the planar case was handled.

\begin{proposition}\label{Var_no_drift} 
Assume that $\mu =0$ and $\bbE[\lVert Y_1\rVert^{mp}]<\infty$ for some $p>2$ .
Suppose that
 $\Prob(\bY_1\in \mathbb{h})=0$ for any affine hyperplane $\mathbb{h}\subset\R^d$.
Then, for each $m\in\{1,\dots,d\}$, 
		\begin{equation*}
		\Var\left(V_m \bigl(\mathrm{conv}\, \bX[0,1] \bigr)\right)<\infty
		\end{equation*} 
		and
		\begin{equation*}
		\lim_{n\to\infty}\frac{\Var(V_m(n))}{n^m}=\Var\left(V_m \bigl(\mathrm{conv}\, \bX[0,1] \bigr)\right).
		\end{equation*}
		\end{proposition}
		\begin{proof}
		Proceeding analogously as in the proof of Theorem \ref{m>1}
	we easily show that there is $c=c(m,d,p)>0$ such that
	\begin{align*}
	V_m(n)^p\le c\!\!\!\! \sum_{i_1,\dots,i_m\in\{1,\dots,d\}}\prod_{j=1}^m \max_{0\le k\le n}|\langle\bS(k), e_{i_j}\rangle|^p.
	\end{align*}
	H\"{o}lder's inequality implies
	\begin{align*}
	\mathbb{E}\Big[\prod_{j=1}^m \max_{0\le k\le n}|\langle\bS(k), e_{i_j}\rangle|^p\Big]\le \prod_{j=1}^m \big(\mathbb{E}\bigl[\bigl(\max_{0\le k\le n}|\langle\bS(k), e_{i_j}\rangle|\bigr)^{mp}\bigr]\big)^{1/m},
	\end{align*}
	and according to \cite[Lemma A.1]{Wade-Xu_SPA} we obtain
	$$\left( \mathbb{E}\bigl[\bigl(\max_{0\le k\le n}|\langle\bS(k), e_i\rangle|\bigr)^{mp}\bigr]\right)^{1/m}=
	\mathcal{O}(n^{p/2}).
	$$ 
	From this it follows that
	$$\mathbb{E}\left[\left(\frac{V_m(n)^2}{n^m}\right)^{p/2}\right]=\frac{1}{n^{mp/2}}\mathbb{E}\left[V_m(n)^{p}\right]=
	\mathcal{O}(1)$$ 
This means that the sequence $\{V_m(n)^2/n^m\}_{n\in\N}$ is uniformly integrable. Thus, in view of Proposition \ref{FCLT-HULL} and \cite[Lemma 3.11]{Kallenberg} we infer the result. 
\end{proof}

We present a corresponding result for random walks with drift where the scaling is of higher order and we cover the case of $m\in \{2,\ldots ,d\}$. The case of planar random walks was studied in \cite{Wade-Xu_PAMS} and \cite[Proposition 3.6]{Wade-Xu_SPA}.
\begin{proposition}\label{Var_drift}
Assume that $\mu\neq 0$ and $\mathbb{E}[\lVert \bY_1\rVert^{ mp}]<\infty$ for some $p>2$. Suppose that $\Prob(\bY_1\in \mathbb{h})=0$ for any affine hyperplane $\mathbb{h}\subset\R^d$. Then, for $m\in\{2,\dots,d-1\}$, 
		\begin{equation*}
		\Var\left(\int_{\mathrm{SO}(d)}\mathrm{Vol}_m\bigl(\mathrm{conv}\, \widetilde\bX_{Q,m}[0,1]\bigr)\nu(\mathrm{d}Q)\right)<\infty
		\end{equation*} 
		and
		\begin{equation*}
		\lim_{n\to\infty}\frac{\Var(V_m(n))}{n^{m+1}}=\binom{d}{m}^2\frac{\kappa^2_d}{\kappa^2_m\kappa^2_{d-m}}\Var\left(\int_{\mathrm{SO}(d)}\mathrm{Vol}_m\bigl(\mathrm{conv}\, \widetilde\bX_{Q,m}[0,1]\bigr)\nu(\mathrm{d}Q)\right).
		\end{equation*}
\end{proposition}

\begin{proof}
By Theorem \ref{CONV-DRIFT-2},
	$$
	\frac{\mathrm{Vol}_m\bigl(P_mQ\, \mathrm{conv}\{\bS(0),\ldots,\bS(n)\}\bigr)}{n^{(m+1)/2}}\xrightarrow[n\nearrow\infty]{\mathcal{D}} \mathrm{Vol}_m\bigl(\mathrm{conv}\, \widetilde\bX_{Q,m}[0,1]\bigr).
	$$ 	 
	Kubota's formula \eqref{Kubota} and continuous mapping theorem then imply 
	$$
	\frac{V_m(n)^2}{n^{m+1}} \xrightarrow[n\nearrow\infty]{\mathcal{D}}\left(\binom{d}{m}\frac{\kappa_d}{\kappa_m\kappa_{d-m}}\int_{\mathrm{SO}(d)}\mathrm{Vol}_m\bigl(\mathrm{conv}\, \widetilde\bX_{Q,m}[0,1]\bigr)\nu(\mathrm{d}Q)\right)^2.
	$$
	Hence, according to \cite[Lemma 3.11]{Kallenberg} it is enough to prove that the sequence $\{V_m(n)^2/n^{m+1}\}_{n\in\N}$ is uniformly integrable. By replacing the standard orthonormal basis with the basis $\{\bar e_1,\dots,\bar e_d\}$ such that $\bar e_1=\mu/\lVert\mu \rVert$ and reasoning as in the proof of Proposition \ref{Var_no_drift}, we obtain
	\begin{align*}
	V_m(n)^p\le c \sum_{i_1,\dots,i_m\in\{1,\dots,d\}}\prod_{j=1}^m \max_{0\le k\le n}|\langle\bS(k), \bar e_{i_j}\rangle|^p,
	\end{align*}
	for some $c=c(m,d,p)>0$.
	Due to \cite[Lemma A.1]{Wade-Xu_SPA}, $$\big(\mathbb{E}\bigl[\bigl(\max_{0\le k\le n}|\langle\bS(k), \bar e_{i}\rangle|\bigr)^{mp}\bigr]\big)^{1/m}=\begin{cases}
	\mathcal{O}(n^{p}), &  i=1,\\
	\mathcal{O}(n^{p/2}), &  i\in\{2,\dots,d\}.
	\end{cases}$$ Hence,	H\"{o}lder's inequality then implies
	\begin{align*}
	\mathbb{E}\Big[\prod_{j=1}^m \max_{0\le k\le n}|\langle\bS(k), \bar e_{i_j}\rangle|^p\Big]\le \prod_{j=1}^m \big(\mathbb{E}\bigl[\bigl(\max_{0\le k\le n}|\langle\bS(k), \bar e_{i_j}\rangle|\bigr)^{mp}\bigr]\big)^{1/m}=
	\mathcal{O}(n^{(m+1)p/2}),
	\end{align*}
that is
	$$\mathbb{E}\left[\left(\frac{V_m(n)^2}{n^{m+1}}\right)^{p/2}\right]=\frac{1}{n^{(m+1)p/2}}\mathbb{E}\left[V_m(n)^{p}\right]=
	\mathcal{O}(1).$$
	This yields uniform integrability of 
	the process $\{V_m(n)^2/n^{m+1}\}_{n\in\N}$.
\end{proof}

\begin{remark}
(i) We could apply the same arguments as in Propositions \ref{Var_no_drift} and \ref{Var_drift} to show  convergence of moments of the appropriately rescaled sequence $\{V_m(n)\}_{n\in\N}$. More precisely, if there is  $p>1$ such that $\bbE[\lVert Y_1\rVert^{mp}]<\infty,$ then  in the case when $\mu=0$,
		$$\bbE\left[V_m \bigl(\mathrm{conv}\, \bX[0,1] \bigr)^p\right]<\infty\qquad \text{and}\qquad 
		\lim_{n\to\infty}\frac{\bbE[V_m(n)^p]}{n^{mp/2}}=\bbE\left[V_m \bigl(\mathrm{conv}\, \bX[0,1] \bigr)^p\right].$$
		If $\mu\neq0$, then
			$$\bbE\left[\left(\int_{\mathrm{SO}(d)}\mathrm{Vol}_m\bigl(\mathrm{conv}\, \widetilde\bX_{Q,m}[0,1]\bigr)\nu(\mathrm{d}Q)\right)^p\right]<\infty$$ and
		$$\lim_{n\to\infty}\frac{\bbE[V_m(n)^p]}{n^{(m+1)p/2}}=\binom{d}{m}^p\frac{\kappa^p_d}{\kappa^p_m\kappa^p_{d-m}}\bbE\left[\left(\int_{\mathrm{SO}(d)}\mathrm{Vol}_m\bigl(\mathrm{conv}\, \widetilde\bX_{Q,m}[0,1]\bigr)\nu(\mathrm{d}Q)\right)^p\right].$$
		(ii) By Theorem \ref{CONV-DRIFT-2}, if $\mu\neq0$, we obtain 
		$$
		\frac{V_d(n)}{n^{(d+1)/2}}\xrightarrow[n\nearrow\infty]{\mathcal{D}} V_d\bigl(\mathrm{conv}\,\widetilde\bX[0,1]\bigr).
		$$ Thus, since $\{V_d(n)^2/n^{d+1}\}_{n\in\N}$ is uniformly integrable, we have 
		$$\Var\left(V_d\bigl(\mathrm{conv}\,\widetilde\bX[0,1]\bigr)\right)<\infty\qquad \text{and} \qquad \lim_{n\to\infty}\frac{\Var(V_d(n))}{n^{d+1}}=\Var\left(V_d\bigl(\mathrm{conv}\,\widetilde\bX[0,1]\bigr)\right).$$
		In particular,
		$$\Var\left(\int_{\mathrm{SO}(d)}\mathrm{Vol}_d\bigl(\mathrm{conv}\, \widetilde\bX_{Q,d}[0,1]\bigr)\nu(\mathrm{d}Q)\right)=\Var\left(V_d\bigl(\mathrm{conv}\,\widetilde\bX[0,1]\bigr)\right).$$ 
		(iii) By Theorem \ref{V_1_p_drift} (see also Theorem \ref{m=1}), we have 
		$$\frac{V_1(n)}{n}
		\xrightarrow[n\nearrow\infty]{\mathbb{P}\text{-a.s.}} \Vert \mu\Vert,$$ which, together with uniform integrability of $\{V_1(n)^2/n^{2}\}_{n\in\N}$, implies $$\lim_{n\to\infty}\frac{\Var(V_1(n))}{n^{2}}=0.$$ 
\end{remark}	
	
In \cite[Theorem 1.1]{Wade-Xu_PAMS} the authors established the variance asymptotics of the perimeter (mean width) of the convex hull for any planar random walk with non-zero drift and finite second moment. The appropriate scaling turns out to be of linear order.  The perimeter of the convex hull of the random walk corresponds to the sequence $\{V_1(n)\}_{n\in\N}$ in higher dimensions and we therefore   
conjecture  that the appropriate scaling for the sequence  $\{\Var(V_1(n))\}_{n\in\N}$ in every dimension $d\ge 1$ is of order $n$ as well.	We finish the article with a partial result in this direction which can be viewed as an extension of \cite[Theorem 2.3]{Snyder-Steele} to higher dimensions.

\begin{proposition}\label{Var_m=1}
	Assume that 
	$\mathbb{E}[\lVert \bY_1\rVert^{2}]<\infty$. Then
	$$\Var(V_1(n))\le n\,\bbE[\lVert\bY_1-\mu\rVert^2],\qquad n\in \bbN.$$
\end{proposition}
\begin{proof}
We follow the approach from \cite[Theorem 2.3]{Snyder-Steele}.	
We clearly have
\begin{align*}
s_{\mathrm{conv}\{\bS(0),\dots,\bS(n)\}}(\theta)
=
\max_{0\le k\le n}\langle \bS(k),\theta \rangle ,\qquad \theta \in \mathbb{S}^{d-1}.
\end{align*}
Thus, by \eqref{MW} and \eqref{V_1=mean_width}, 
$$
V_1(n)
=
\frac{1}{\kappa_{d-1}}\int_{\mathbb{S}^{d-1}}\max_{0\le k\le n}\langle \bS(k),\theta \rangle \, \sigma (\mathrm{d}\theta).
$$
	Further, let $\{\bY'_i\}_{i\in\N}$ be an independent copy of $\{\bY_i\}_{i\in\N}$. For $i\in\N$ we  set \begin{equation}\label{S'}\bS^i(n)=
	\begin{cases}
	\bS(n), &  n<i,\\
	\bS(n)-\bY_i+\bY_i', &  i\ge n,
	\end{cases}
	\end{equation}
	and 
	$$
	V_1^i(n)=\frac{1}{\kappa_{d-1}}\int_{\mathbb{S}^{d-1}}\max_{0\le k\le n}\langle \bS^i(k),\theta \rangle \, \sigma (\mathrm{d}\theta).
	$$
	 According to \cite[remark on page 755]{Steele}, \begin{equation}\begin{aligned}\label{var}\Var(V_1(n))&\le \frac{1}{2}\sum_{i=1}^n\bbE\left[ \bigl(V_1(n)-V_1^i(n)\bigr)^2\right]\\&=\frac{1}{2\kappa^2_{d-1}}\sum_{i=1}^n\bbE\left[ \left(\int_{\mathbb{S}^{d-1}}\bigl(\max_{0\le k\le n}\langle \bS^i(k),\theta \rangle-\max_{0\le k\le n}\langle \bS(k),\theta \rangle\bigr) \sigma (\mathrm{d}\theta)\right)^2\right].\end{aligned}\end{equation}
	From \eqref{S'}, for any $\theta \in \mathbb{S}^{d-1}$, we have \begin{align*}\langle \bS^i(k),\theta \rangle&=
	\begin{cases}
	\langle \bS(k),\theta\rangle, &  k<i,\\
	\langle \bS(k),\theta\rangle-\langle\bY_i,\theta\rangle+\langle\bY_i',\theta\rangle, &  i\ge k
	\end{cases}\\&\le \langle\bS(k),\theta\rangle+|\langle\bY_i,\theta\rangle-\langle\bY_i',\theta\rangle|.
	\end{align*}
	Hence, $$\max_{0\le k\le n}\langle\bS^i(k),\theta\rangle-\max_{0\le k\le n}\langle\bS(k),\theta\rangle\le |\langle\bY_i,\theta\rangle-\langle\bY_i',\theta\rangle|.$$ 
	By symmetry, we can replace the left-hand side of the above inequality with its absolute value.
	This together with  \eqref{var} implies
	\begin{align*}\Var(V_1(n))&\le \frac{1}{2\kappa^2_{d-1}}\sum_{i=1}^n\bbE\left[ \left(\int_{\mathbb{S}^{d-1}}|\langle\bY_i,\theta\rangle-\langle\bY_i',\theta\rangle| \sigma (\mathrm{d}\theta)\right)^2\right]\\
	&\le \frac{1}{2\kappa_{d-1}}\sum_{i=1}^n \int_{\mathbb{S}^{d-1}}\bbE\left[|\langle\bY_i-\bY_i',\theta\rangle|^2\right] \sigma (\mathrm{d}\theta)\\
	&\le \frac{n\, \bbE[\lVert\bY_1-\mu\rVert^2]}{\kappa_{d-1}} \int_{\mathbb{S}^{d-1}}\lVert\theta\rVert^2 \sigma (\mathrm{d}\theta)\\
	&\le n\,\bbE[\lVert\bY_1-\mu\rVert^2],\end{align*}
	and the proof is finished.
\end{proof}

\section*{Acknowledgement} 
\noindent
We thank M.\ Puljiz for discussions and ideas related to the proof of Theorem \ref{ST}.
This work has been supported by \textit{Deutscher Akademischer Austauschdienst} (DAAD) and \textit{Ministry of Science and Education of the Republic of Croatia} (MSE) via project \textit{Random Time-Change and Jump Processes}.
Financial support through the  
\textit{Alexander-von-Humboldt Foundation} under project No.\ HRV 1151902 HFST-E and \textit{Croatian Science Foundation} under project 8958 (for N.\ Sandri\'c), and \textit{Croatian Science Foundation} under project 4197 (for S.\ \v Sebek) is gratefully acknowledged.

\bibliographystyle{abbrv}
\bibliography{HULL}

\end{document}